\documentclass[10pt,dvipdfmx,a4paper,reqno]{article}
\usepackage{amsmath,enumerate,bm,bbm,color,geometry}
\usepackage[amsmath,thmmarks]{ntheorem}
\theoremstyle{change}
\theoremseparator{.}
\usepackage{amssymb}
\usepackage{amsbsy}
\usepackage{amsfonts}
\usepackage{amstext}
\usepackage{amscd}
\usepackage[dvipdfmx]{graphicx}

\providecommand{\U}[1]{\protect\rule{.1in}{.1in}}

\theoremstyle{plain}
\newtheorem{thm}{Theorem}[section]

\newtheorem{cor}[thm]{Corollary}
\newtheorem{lem}[thm]{Lemma}

\theoremstyle{definition}

\theoremstyle{remark}
\def\proofsymbol{\rule{0.5em}{0.5em}}
\theoremheaderfont{\normalfont\bfseries}
\theorembodyfont{\normalfont}
\newtheorem{rem}[thm]{Remark}

\newtheorem{examples}[thm]{Examples}
\theoremstyle{nonumberplain}
\theoremsymbol{\enspace\ensuremath{{\proofsymbol}}}
\newtheorem{proof}{Proof}

\theoremstyle{empty}
\newtheorem{proofof}{}

\newcommand{\C}{\mathbb{C}}
\newcommand{\R}{\mathbb{R}}

\newcommand{\N}{\mathbb{N}}

\newcommand{\calC}{\mathcal{C}}
\newcommand{\calR}{\mathcal{R}}
\newcommand{\calUI}{\mathcal{UI}}

\newcommand{\DD}{\mathbf{D}}
\def\im{{\sf Im}}
\def\re{{\sf Re}}

\newcommand\sign{{\sf sign}}
\renewcommand{\leqq}{\leqslant}
\renewcommand{\geqq}{\geqslant}

\begin{document}
\title{The normal distribution is freely selfdecomposable}
\author{Takahiro Hasebe\\Department of Mathematics,\\ Hokkaido University \\thasebe@math.sci.hokudai.ac.jp
\and Noriyoshi Sakuma\\ Department of Mathematics,\\ Aichi University of Education \\sakuma@auecc.aichi-edu.ac.jp
\and Steen Thorbj{\o}rnsen\\ Department of Mathematics,\\ University of Aarhus\\steenth@math.au.dk}
\date{\today}
\maketitle

\begin{abstract}
The class of selfdecomposable distributions in free probability theory
was introduced by Barndorff-Nielsen and the third named author. 
It constitutes a fairly large subclass of the freely infinitely divisible
distributions, but so far specific examples have been limited to
Wigner's semicircle distributions, the free stable distributions,
two kinds of free gamma distributions and a few other examples.
In this paper, we prove that the (classical) normal
distributions are freely selfdecomposable. More generally it is
established that the Askey-Wimp-Kerov distribution $\mu_c$ is freely
selfdecomposable for any $c$ in $[-1,0]$. The main ingredient in the proof
is a general characterization of the freely selfdecomposable distributions in
terms of the derivative of their free cumulant transform.
\end{abstract}


\smallskip


\section{Introduction}
Infinitely divisible distributions and L\'evy processes have
constituted a major role in the development of probability theory for
more than eighty years (see \cite{Sato99} for some main aspects). 
Following Voiculescu's foundation of free probability theory in the
early 1980's he further introduced the class of infinitely
divisible distributions with respect to free additive convolution
$\boxplus$ (see \cite{VDN,BeVo1993}). We denote this class by
$I(\boxplus)$, and refer to its members as freely infinitely divisible
(FID) distributions. As in classical probability the FID distributions
can be characterized as those admitting a L{\'e}vy-Khintchine representation
of the free analog of the cumulant transform. This was established by 
Bercovici and Voiculescu in \cite{BeVo1993}. Specifically
the free cumulant transform $\calC_\mu$ of a (Borel-) probability measure
$\mu$ on $\R$ is defined in terms of its
Cauchy-Stieltjes transform $G_\mu$ given by 
\begin{equation*}
G_\mu(z)=\int_{\R}\frac{1}{z-t}\,\mu(dt), \qquad(z\in\C^+),
\end{equation*}
where $\C^+$ (resp.\ $\C^-$) denotes the set of complex numbers with
strictly positive (resp.\ strictly negative) imaginary part.
Note in particular that $\im(G_\mu(z))<0$ for any $z$ in $\C^+$, and
hence we may consider the reciprocal Cauchy transform
$F_\mu\colon\C^{+}\to\C^{+}$ given by $F_{\mu}(z)=1/G_{\mu}(z)$ for $z$ in $\C^+$.
For any probability measure $\mu$ on $\R$ and any $\lambda$ in
$(0,\infty)$ there exist 
positive numbers $\alpha,\beta$ and $M$ such that $F_{\mu}$ is univalent  
on the set $\Gamma_{\alpha,\beta}:=\{z \in \C^{+} \,|\, \im(z) >\beta,
|\re(z)|<\alpha \im(z)\}$ and such that 
$F_{\mu}(\Gamma_{\alpha,\beta})\supset\Gamma_{\lambda,M}$. 
Therefore the right inverse $F^{-1}_{\mu}$ of $F_{\mu}$ exists on
$\Gamma_{\lambda,M}$, and the free cumulant transform
$\calC_\mu$ can be defined by 
\begin{align}
\calC_{\mu}(w) =wF^{-1}_{\mu}(1/w)-1, \quad\text{for all $w$ such that
  $1/w \in \Gamma_{\lambda,M}$}.
\label{def_Cmu_eq}
\end{align}
The name refers to the fact that $\calC_\mu$ linearizes free additive
convolution (cf.\ \cite{BeVo1993}). Variants of $\calC_\mu$ (with the
same linearizing property) are the $R$-transform $\calR_\mu$ and the
Voiculescu transform $\varphi_\mu$ related by the following equalities:
\begin{equation}
\calC_\mu(w)=w\calR_\mu(w)=w\varphi_\mu(\tfrac{1}{w}).
\label{relations_eq}
\end{equation}
The free version of the L\'evy-Khintchine representation now amounts
to the statement that a probability measure $\mu$ on $\R$
is in $I(\boxplus )$, if and only if there exist $a\geqq 0$,
$\eta\in\R$ and a L{\'e}vy measure\footnote{A (Borel-) measure $\nu$
  on $\R$ is called a L\'evy measure, if $\nu(\{0\})=0$ and
  $\int_{\R}\min\{1,x^2\}\,\nu(dx)<\infty$.}
$\nu$ such that  
\begin{align}
\calC_{\mu}(w) = a w^{2}+\eta w +
\int_{\R}\left(\frac{1}{1- w x}-1-w x 1_{[-1,1]}(x)\right)\nu(dx).
\label{eqno1}
\end{align}
The triplet $(a,\eta,\nu)$ is uniquely determined and referred to as
the \emph{free characteristic triplet} for $\mu$, and $\nu$ is referred to as the \emph{free L\'evy measure} for $\mu$. In terms of the Voiculescu
transform $\varphi_\mu$ the free L\'evy-Khintchine representation takes
the form:
\begin{equation}
\label{eqno1a}
\varphi_{\mu}(z)=\gamma+\int_{{\mathbb R}}\frac{1+tz}{z-t}\, \sigma(dt),
\qquad (z\in{\mathbb C}^+),
\end{equation}
where the \emph{free generating pair} $(\gamma,\sigma)$ is uniquely
determined and related to the free characteristic triplet by the
formulas:
\begin{equation}
\begin{split}
\nu({\rm d}t)&=\frac{1+t^2}{t^2}\cdot 1_{{\mathbb R}\setminus\{0\}}(t) \
\sigma({\rm d}t),\\ 
\eta&=\gamma+\int_{{\mathbb R}}t\Big(1_{[-1,1]}(t)-\frac{1}{1+t^2}\Big) \
\nu({\rm d}t), \\
a&=\sigma(\{0\}).
\end{split}
\label{eqno1b}
\end{equation}
In particular $\sigma$ is a finite measure.
The right hand side of \eqref{eqno1a} gives rise to an analytic
function defined on all of $\C^+$, and in fact the property that
$\varphi_\mu$ can be extended analytically to all of $\C^+$ also
characterizes the measures in $I(\boxplus)$. More precisely Bercovici
and Voiculescu established in \cite{BeVo1993} the following
fundamental result:

\begin{thm}
A probability measure $\mu$ on $\R$ is in $I(\boxplus)$, if and only if
the Voiculescu transform $\varphi_{\mu}$ has an analytic extension
defined on $\C^{+}$ with values in $\C^{-}\cup \R$.
\label{BV_char_of_FID}
\end{thm}  

Research on FID distributions developed rapidly since 1999, when
Bercovici and Pata introduced and studied a natural bijection between
the classes of classically and freely infinitely divisible distributions
(see \cite{BePa1999} and \cite{BNTh2004}). As a natural step in this
development the class of \emph{freely selfdecomposable} (FSD) distributions
was introduced in \cite{BNTh2002}.
A probability distribution $\mu$ on $\R$ is said to be FSD, if, for any
$c$ in $(0,1)$ there  exists a probability measure $\rho_{c}$ such that
$\mu=\DD_{c}(\mu)\boxplus \rho_{c}$, where $\DD_c(\mu)$ denotes the
scaling of $\mu$ by the constant $c$. We denote the set of all freely
selfdecomposable distributions by $L(\boxplus )$.
Chistyakov and Goetze \cite[Theorem 2.8]{CG2008} identified the class $L(\boxplus)$ with the set of possible weak limits of 
\begin{equation}\label{SDL}
\delta_{a_n} \boxplus \DD_{b_n}(\mu_1 \boxplus \mu_2 \boxplus \cdots \boxplus \mu_n), \qquad n=1,2,3,\dots, 
\end{equation}
where $a_n\in\R, b_n>0$ and $\mu_1,\mu_2,\dots$ are probability
measures on $\R$ such that $\{\DD_{b_n}(\mu_k)\}_{1\leq k \leq n,
  1\leq n}$ forms an infinitesimal array. This is in complete analogy
with the classical limit theorem for (classically) selfdecomposable
distributions (see e.g.\ the book of Gnedenko and Kolmogorov
\cite{GK68}). 

If $\mu$ is FSD then $\mu$ is automatically FID (see \cite{BNTh2002}),
and therefore has a L{\'e}vy-Khintchine representation. 
The FSD distributions can then (in full analogy with selfdecomposability in
classical probability) be characterized as the FID measures for
which the free L\'evy measure $\nu$ (appearing in the free characteristic
triplet) takes the form: 
\begin{equation}
\nu(dx)=\frac{k(x)}{|x|}1_{\R\setminus\{0\}}(x)\, dx,
\label{eqno2}
\end{equation}
where the function $k\colon\R\setminus\{0\}\to[0,\infty)$
is non-decreasing on $(-\infty,0)$ and non-increasing on $(0,\infty)$. 
From this characterization one can readily list a number of examples of
FSD distributions.

\begin{examples}
\begin{enumerate}[\rm(i)]

\item For any $a$ in $\R$ and $r$ in $(0,\infty)$ the semi-circle
  distribution centered at $a$ and of radius $r$ is the probability
  measure $\gamma_{a,r}$ given by
\[
\gamma_{a,r}(dt)=\frac{2}{\pi r^2}\sqrt{r^2-(t-a)^2}1_{[a-r,a+r]}(t)\,dt.
\]
These distributions are freely selfdecomposable, as $\gamma_{a,r}$ has
free characteristic triplet $(\frac{r^2}{4},a,0)$. 

\item The free stable distributions with index $\alpha \in (0,2)$ are FSD, as
they have free characteristic triplets $(0,\eta,\nu)$, where $\nu$ has the
form \eqref{eqno2} with 
$$
k(x)=cx^{-\alpha}1_{(0,\infty)}(x)+c'|x|^{-\alpha} 1_{(-\infty,0)}(x),
$$
and $c,c'$ are parameters in $[0,\infty)$. The main distributional
properties of the free stable 
distributions were uncovered by Biane in the appendix to \cite{BePa1999}.

\item
The free Meixner distributions have been studied intensely by e.g.\
Saitoh and Yoshida \cite{SaYo2001}, Anshelevich \cite{An2003} 
and Bryc and Bo{\.z}ejko \cite{BB2006}. In \cite{BB2006} these
distributions are 
introduced as the two-parameter family 
$\{\mu_{a,b}\mid a\in\R, b\geqq -1\}$
of probability measures with Cauchy-Stieltjes transforms given by
\[
G_{\mu_{a,b}}(z)=\frac{(1+2b)z+a-\sqrt{(z-a)^2-4(1+b)}}{2(bz^2+az+1)},
  \qquad(z\in\C^+). 
\]
More generally all \emph{increasing} affine transformations of the
measures $\mu_{a,b}$ 
are also referred to as free Mexiner distributions. It was shown in
\cite{SaYo2001} that $\mu_{a,b}$ is
$\boxplus$-infinitely divisible when $b\geqq0$. 
If $a=b=0$, $\mu_{a,b}$ is a semi-circle distribution
and hence FSD. The case $b=0, a\ne0$ corresponds to the free Poisson
distributions, which are not FSD (see (vi) below). If $b>0$,
the free L\'evy measure for $\DD_c(\mu_{a,b})$ is given by
\begin{align*}
\nu(dx)=\frac{1}{2\pi b}
\frac{\sqrt{4bc^2-(x-ca)^2}}{x^2}1_{[ca-2c\sqrt{b},ca+2c\sqrt{b}]}(x)\,dx 
\end{align*}
for any positive number $c$.
Elementary calculus shows that the function
\[
k(x)=\frac{1}{2\pi b}
\frac{\sqrt{4bc^2-(x-ca)^2}}{|x|}1_{[ca-2c\sqrt{b},ca+2c\sqrt{b}]}(x)
\]
satisfies the monotonicity property described in \eqref{eqno2}, if
and only if $4b\geqq a^{2}$. Thus $\DD_c(\mu_{a,b})$ is FSD, if and
only if $4b\geqq a^2$. In case this inequality is strict, $\mu_{a,b}$ is
termed a pure free Meixner law in \cite{BB2006}, whereas the case
$4b=a^2$ is referred to as a free gamma distribution. 

\item P\'erez-Abreu and Sakuma \cite{PAS} introduced another type of free
  gamma distributions, namely the images of the classical gamma distributions
  under the Bercovici-Pata bijection. They have free L\'evy measure in
  the form: 
$$
\nu(dx) = \frac{ce^{-\alpha x}}{x} 1_{(0,\infty)}(x)\, dx, 
$$
where $\alpha$ and $c$ are positive parameters.
As the function $x\mapsto ce^{-\alpha x}$ is non-increasing on $(0,\infty)$,
these free gamma distributions are also FSD. Their main distributional
properties were uncovered by Haagerup and Thorbj{\o}rnsen in \cite{HT7}.

\item The Student t-distribution with 3 degrees of freedom is
the probability measure given by the Lebesgue density
\[
f(t)=\frac{2}{\pi\sqrt{3}}\Big(1+\frac{t^2}{3}\Big)^{-2},
\qquad(t\in\R).
\]
In the recent paper \cite{HaSa2016} it was found that this
distribution is FSD.

\item For $\lambda$ in $(0,\infty)$ and $\alpha$ in
  $\R\setminus\{0\}$ the free 
  Poisson distribution with parameters $(\lambda,\alpha)$ is the
  probability measure $\mu_{\lambda,\alpha}$ given by
\[
\mu_{\lambda,\alpha}(dt)=(1-\lambda)^+\delta_0(dt)
+\frac{1}{2\pi|\alpha|t}
\sqrt{4\lambda\alpha^2-(t-\alpha(1+\lambda))^2}
1_{[(1-\sqrt{\lambda})^2,(1+\sqrt{\lambda})^2]}(\alpha^{-1}t)\,dt
\]
(see e.g.\ \cite{NiSpBook}). This distribution is FID but not FSD,
since its free L{\'e}vy measure is $\nu(dt)=\lambda\delta_{\alpha}(dt)$.
Note that, in some contexts,
the free Poisson distributions are also referred to as free gamma
distributions (not to be mistaken with the two classes described above).  
\end{enumerate}
\end{examples}

The examples above illustrate the general fact that all FSD
distributions are unimodal (in full analogy with classical probability
theory). This was established in \cite{HaTh2015}.

Triggered by a question of P\'erez-Abreu, it was recently proved by
Belinschi et al.\ (see \cite{BBLS2011})
that the classical normal (or Gaussian) distributions
are FID. The proof is based on the characterization of $I(\boxplus)$ in 
Theorem~\ref{BV_char_of_FID}.
As a natural follow-up question Marek Bo{\.z}ejko asked
whether the normal distributions are FSD or not. 
In order to answer Bo{\.z}ejko's question (in the positive), 
this paper establishes a characterization of the free cumulant
transform of FSD distributions akin to Theorem~\ref{BV_char_of_FID}
(see Theorem~\ref{main1} below).
Based on some facts about the Voiculescu transform of the normal
distribution, established in \cite{BBLS2011}, and a fundamental
theorem due to Kerov (see Theorem~\ref{Kerovs_Thm}),
we can subsequently argue
that the normal distributions satisfy this characterization. 
More generally we prove, using the same method, that the Askey-Wimp-Kerov
distribution $\mu_c$ is FSD for any $c$ in $[-1,0]$. Let us recall
here (see e.g.\ \cite{Ke1998}) that for any $c$ in $(-1,\infty)$ the
Askey-Wimp-Kerov distribution $\mu_c$ is the measure on $\R$ with Lebesgue
density 
\[
\kappa_c(t)=\frac{1}{\sqrt{2\pi}\Gamma(c+1)}|D_{-c}(it)|^{-2},
\qquad(t\in\R),
\]
where $D_{-c}(z)$ is the solution to the differential equation:
\[
\frac{d^2y}{dz^2}+\Big(\frac{1}{2}-c-\frac{z^2}{4}\Big)y=0,
\]
satisfying the initial conditions:
\[
D_{-c}(0)=\frac{\Gamma(\frac{1}{2})2^{-c/2}}{\Gamma(\frac{1+c}{2})},
\quad\text{and}\quad
D_{-c}'(0)=\frac{\Gamma(-\frac{1}{2})2^{-(c+1)/2}}{\Gamma(\frac{c}{2})}.
\]
When $c>0$, the solution $D_{-c}$ has the integral representation 
\[
D_{-c}(z) = \frac{e^{-z^2/4}}{\Gamma(c )} \int_0^\infty e^{-z x}
x^{c-1}e^{-x^2/2} \,dx.
\]
It was proved in \cite{AsWi1984} that for any $c$ in $(-1,\infty)$ the
measure $\mu_c$ is a probability measure. The case $c=0$ corresponds
to the standard Gaussian distribution $N(0,1)$, and the family
$(\mu_c)_{c\in(-1,\infty)}$ can be extended 
continuously at $-1$ by defining $\mu_{-1}$ to be the Dirac
point mass $\delta_0$ at 0. Then for all $c$ in $[-1,\infty)$ 
the Cauchy-Stieltjes transform $G_{\mu_c}$ has the continued fraction
expansion:
\[
G_{\mu_c}(z)=\cfrac{1}{z-\cfrac{c+1}{z-\cfrac{c+2}{z-\cfrac{c+3}{z-\cdots}}}}, 
\]
or, equivalently, the orthogonal polynomials $(H_n(x;c))_{n\in\N_0}$ with respect
to $\mu_c$ are given by the recurrence relation:
\[
H_{n+1}(x;c)=xH_n(x;c)-(c+n)H_{n-1}(x;c), \qquad(n\geqq1),
\]
with $H_0(x,c)=1$ and $H_1(x;c)=x$. In the case $c=0$, one recovers
the Hermite polynomials (the orthogonal polynomials with respect to
$N(0,1)$), and for general $c$ the polynomials $H_n(x;c)$ are referred
to as associated Hermite polynomials (cf.\ \cite{AsWi1984}). Further
information is available in \cite{AsWi1984,BBLS2011,Ke1998}.

The remaining part of the paper is organized as follows:
In Section~\ref{sec:Proof_of_Main1} we establish the above mentioned
characterization of the free cumulant transforms of FSD
distributions. The proofs of some technical (but rather elementary) lemmas in
this section are deferred to an appendix in order to maintain a steady
flow of the paper. In Section~\ref{sec:Proof_of_Main2}, we prove the
free selfdecomposability of the Askey-Wimp-Kerov distribution $\mu_c$
for any $c$ in $[-1,0]$, and as an immediate corollary we conclude
that all normal distributions are freely selfdecomposable.

\section{A characterization of free selfdecomposability 
in terms of the free cumulant transform}
\label{sec:Proof_of_Main1} 

In this section we establish a characterization of free
selfdecomposability akin to the characterization of free infinite
divisibility in Theorem~\ref{BV_char_of_FID}.
To prove this result (Theorem~\ref{main1} below), we first need to
establish some lemmas. The first four lemmas below are rather elementary,
but for completeness we include proofs of Lemma~\ref{lem1b}, 
Lemma~\ref{lem4} and Lemma \ref{lem1} in the appendix. A proof of
Lemma~\ref{lem2} can be found in e.g.\ \cite[page~150]{Fe}. 

Throughout the paper $\log(z)$ denotes the usual (real-valued)
logarithm of $z$, whenever $z$ is a positive real number. When $z$ is
a complex number, the relevant branch of the logarithm will be
specified, if it is not clear from the context.

\begin{lem}\label{lem1b} Let $a,b$ be real numbers, such that $a<b$,
  and let $f\colon[a,b]\to\R$ be a continuous function. Consider
  further the standard argument function
  $\arg\colon\C\setminus\{iy\mid
  y\in(-\infty,0]\}\to(-\frac{\pi}{2},\frac{3\pi}{2})$. Then the
  following assertions hold:

\begin{enumerate}[\rm(i)]

\item As  $v\downarrow0$ we have that
\[
\tfrac{1}{2}\int_{a}^{b}f(x)\log((x-u)^{2}+v^{2})\,dx
\longrightarrow \int_{a}^{b}f(x)\log(|x-u|)\,dx
\qquad\text{uniformly w.r.t.\ $u\in[a,b]$.} 
\]

\item For any anti-derivative $F$ of $f$, we have that
\[
\int_{a}^{b}f(x)\arg (u+iv-x)\,dx\longrightarrow i\pi(F(b)-F(u))
\qquad\text{as $v\downarrow0$, uniformly w.r.t.\ $u\in[a,b]$.}
\]

\item As $u+iv\to0$ from $\C^+$ we have that
\[
\int_a^bf(x)\log((u-x)^2+v^2)\,dx\longrightarrow
\int_a^bf(x)\log(x^2)\,dx.
\]
\end{enumerate}
\end{lem}

\begin{lem}\label{lem2}
Let $\rho$ be a finite Borel measure on $\R$, and let $a,b$ be real
numbers such that $a<b$, and such that $\rho(\{ a\})=\rho(\{b\})=0$. 
Let further $l(x)=\rho((x,\infty))$ for any $x$ in $\R$. Then  
for any $f$ in $C^{1}([a,b])$ we have that
\begin{align*}
\int_{a}^{b}f(x)\,\rho(dx)=-\big[f(x)l(x)\big]_{a}^{b}
+\int_{a}^{b}f'(x)l(x)\,dx. 
\end{align*}
\end{lem}

\begin{lem}\label{lem4}
Let $\rho$ be a Borel measure on $\R$ such that
$\int_{\R}\log(2+|x|)\,\rho(dx)<\infty$. Consider further the
function $k\colon\R\setminus\{0\}\to[0,\infty)$ given by
\[
k(x)=
\begin{cases}
\int_x^\infty\frac{1+y^2}{y^2}\,\rho(dy), &\text{if $x>0$,}
\\
\int_{-\infty}^x\frac{1+y^2}{y^2}\,\rho(dy), &\text{if $x<0$.}
\end{cases}
\]
Then $k$ is increasing on $(-\infty,0)$, decreasing on $(0,\infty)$,
and the following assertions hold:

\begin{enumerate}[\rm(i)]

\item\label{lem4-1} The measure
  $\frac{k(x)}{|x|}1_{\R\setminus\{0\}}(x)\,dx$ is a L\'evy measure.

\item $x^2k(x)\to0$ as $x\to0$.

\item $k(x)\log(|x|)\to0$ as $|x|\to\infty$.

\item\label{lem4-2} For any $z$ in $\C^-$ we have that
\[
\lim_{|x|\to\infty}\Big(\log(1-xz)+\frac{xz}{1+x^2}\Big)k(x)
=0=\lim_{x\to0}\Big(\log(1-xz)+\frac{xz}{1+x^2}\Big)k(x),
\]
where $\log$ is the standard branch of the logarithm on
$\C\setminus(-\infty,0]$.

\end{enumerate}
\end{lem}

\begin{lem}\label{lem1}
Let $a,b$ be real numbers, such that $a<b$, and let $m$ be a positive
integer. Suppose further that $f\colon(a,b)\to\R$ belongs to
$L^1((a,b), dx) \cap C^{m}((a,b)) $. 
Consider also the Cauchy transform of $f$:
\[
G_{f}(z)=\int_{a}^{b}\frac{f(x)}{z-x}\,dx, \qquad(z\in\C^+).
\]
Then $G_f$ and all of its derivatives up to order $m-1$ can be
extended to continuous functions on $\C^{+}\cup(a,b)$.
\end{lem}

\begin{lem}\label{lem3}
Suppose that $k$ is a function in $C^{\infty}(\R\setminus\{0\})$ with
bounded support and such that $k$ and all its derivatives are bounded
functions on $\R\setminus\{0\}$. Suppose in addition that $k$ is
increasing on $(-\infty,0)$ 
and decreasing on $(0,\infty)$, and let $\mu$ be the measure in
$L(\boxplus)$ with free characteristic triplet
$(0,\int_{-1}^1\sign(t)k(t)\,dt,\frac{k(t)}{|t|}\,dt)$.  

Then the free cumulant transform $\calC_\mu$ extends to an analytic
function $\calC_\mu\colon\C^-\to\C$, such that
$\im(\calC_{\mu}'(z))\leqq0$ for any $z$ in $\C^-$.
\end{lem}

\begin{proof} For each $t$ in $\R\setminus\{0\}$ we put
  $\tilde{k}(t)=\sign(t)k(t)$.
Since $\mu\in L(\boxplus)\subseteq I(\boxplus)$, it
follows from Theorem~\ref{BV_char_of_FID} and \eqref{relations_eq}
that $\calC_\mu$ can be extended to the analytic function
$\calC_\mu\colon \C^-\to\C$ given by
\begin{equation*}
\begin{split}
\calC_{\mu}(w)&=w\int_{-1}^1\tilde{k}(t)\,dt
+\int_\R\Big(\frac{1}{1-wt}-1-wt1_{[-1,1]}(t)\Big)\frac{k(t)}{|t|}\,dt 
=\int_\R\Big(\frac{1}{1-wt}-1\Big)\frac{k(t)}{|t|}\,dt
\\[.2cm]
&=w\int_\R\frac{t}{1-wt}\frac{k(t)}{|t|}\,dt
=w\int_\R\frac{\tilde{k}(t)}{1-wt}\,dt
\end{split}
\end{equation*}
for any $w$ in $\C^-$. Setting $w=\frac{1}{z}$ we find for any $z$
in $\C^+$ that
\begin{equation}
\calC_{\mu}\big(\tfrac{1}{z}\big)=\frac{1}{z}
\int_\R\frac{\tilde{k}(t)}{1-\frac{t}{z}}\,dt
=\int_\R\frac{\tilde{k}(t)}{z-t}\,dt
=:G_{\tilde{k}}(z).
\label{lem3_eq1}
\end{equation}
Choosing $n$ in $\N$ such that the support of $k$ is contained in
$[-n,n]$, it follows by application of Lemma~\ref{lem1} to the restrictions
of $\tilde{k}$ to $(-n,0)$ and $(0,n)$ that $G_{\tilde{k}}$ and all
its derivatives can be
extended to continuous functions on $\C^+\cup(-n,0)\cup(0,n)$. Letting
$n\to\infty$, we conclude that
$G_{\tilde{k}}$ and all its derivatives can
be extended to continuous functions on
$\C^+\cup(\R\setminus\{0\})$. From \eqref{lem3_eq1}
we have that 
\begin{equation}
\calC_\mu'(\tfrac{1}{z})
=-z^2G_{\tilde{k}}'(z)
\label{lem3_eq2}
\end{equation}
for any $z$ in $\C^+$. In particular we thus deduce that the
function $z\mapsto\calC_\mu'(1/z)$ can be extended to a continuous
function on $\C^+\cup(\R\setminus\{0\})$, and hence $\calC_\mu'$ can
be extended to a continuous function on
$\C^-\cup(\R\setminus\{0\})$. With $n$ chosen as above, we note further by
dominated convergence that
\[
\calC_\mu'\big(\tfrac{1}{z}\big)
=\int_{-n}^n\frac{z^2}{(z-t)^2}\tilde{k}(t)\,dt
\longrightarrow \int_{-n}^n\tilde{k}(t)\,dt = \int_{\R}\tilde{k}(t)\,dt
\qquad\text{as $|z|\to\infty$, $z\in\C^+\cup\R$.} 
\]
It follows thus that the function $\Psi\colon\C^-\cup\R\to\C$ given by
\begin{equation*}
\Psi(w)=
\begin{cases}
\calC_\mu'(w), &\text{if $w\in\C^-\cup(\R\setminus\{0\})$,}\\
\int_{\R}\tilde{k}(t)\,dt, &\text{if $w=0$,}
\end{cases}
\end{equation*}
is continuous. In addition $\im(\Psi)$ is harmonic on $\C^-$. We shall
argue below that

\begin{itemize}

\item[(a)] $\im(\Psi(x))\leqq 0$ for any $x$ in $\R$.

\item[(b)] $\Psi(w)\to0$, as $|w|\to\infty$, $w\in\C^-\cup\R$.

\end{itemize}

Once (a) and (b) are verified, the proof is completed as follows:
Given any $\epsilon$ in $(0,\infty)$ and $w_0$ in $\C^-$, we
choose $R$ in $(0,\infty)$, such that $R>|w_0|$, and such that
$|\Psi(w)|\leqq\epsilon$ for all $w$ in $\C^-\cup\R$ satisfying that $|w|\geqq
R$. Putting $\gamma_R=\{R e^{i \theta}\mid\theta\in[-\pi,0]\}$, it
follows now by the maximum principle for harmonic functions that
\[
\im\big(\calC_\mu'(w_0)\big)=\im\big(\Psi(w_0)\big)\leqq
\sup\big\{\im(\Psi(w))\bigm|w\in[-R,R]\cup \gamma_R\big\}\leqq\epsilon,
\]
Since $\epsilon$ was arbitrary, we conclude that
$\im(\calC_\mu'(w_0))\leqq0$, as desired.

It remains to verify (a) and (b): Regarding (a) consider a fixed
number $a$ in $(0,\infty)$. Then for any $x$ in $(a,\infty)$ and any
positive integer $n$ it follows from \eqref{lem3_eq2} that
\begin{align*}
G_{\tilde{k}}(x+\tfrac{i}{n})
&=G_{\tilde{k}}(a+\tfrac{i}{n})+\int_a^xG_{\tilde{k}}'(t+\tfrac{i}{n})\,dt
\\
&=G_{\tilde{k}}(a+\tfrac{i}{n})
-\int_a^x(t+\tfrac{i}{n})^{-2}\calC_{\mu}'\big((t+\tfrac{i}{n})^{-1}\big)\,dt 
\xrightarrow[n\to\infty]{}G_{\tilde{k}}(a)
-\int_a^xt^{-2}\calC_{\mu}'\big(t^{-1}\big)\,dt, 
\end{align*}
where the convergence follows e.g.\ by uniform continuity of $z\mapsto
z^{-2}\calC_\mu'(z^{-1})$ on $[a,x]\times(i[0,1])$.
At the same time the method of Stieltjes Inversion yields for
Lebesgue-almost all $x$ in $(a,\infty)$ that
\begin{align*}
\tilde{k}(x)
=-\frac{1}{\pi}\lim_{n\to\infty}\im\big(G_{\tilde{k}}(x+\tfrac{i}{n})\big)
=-\frac{1}{\pi}\im(G_{\tilde{k}}(a))
+\frac{1}{\pi}\int_a^xt^{-2}\im\big(\calC_{\mu}'(t^{-1})\big)\,dt.
\end{align*} 
Since $\tilde{k}$ is continuous, this equality actually holds for all
$x$ in $(a,\infty)$, and hence we further deduce that
\begin{equation}
\tilde{k}'(x)=\frac{1}{\pi}x^{-2}\im\big(\calC_\mu'(x^{-1})\big)
\label{lem3_eq4}
\end{equation}
for all $x$ in $(a,\infty)$. Since $a$ was chosen arbitrarily in
$(0,\infty)$, \eqref{lem3_eq4} holds for all $x$ in $(0,\infty)$ and
by similar argumentation also for all $x$ in $(-\infty,0)$. Thus for
any $x$ in $\R\setminus\{0\}$, we conclude that
$
\im(\calC_\mu'(\tfrac{1}{x}))=\pi x^2\tilde{k}'(x)\leqq0
$
by the definition of $\tilde{k}$ and the assumptions on $k$.

Regarding (b), we show that $\calC_\mu'(\frac{1}{z})\to0$ as $z\to0$,
$z\in\C^+\cup\R\setminus\{0\}$. We note initially that 
\[
\calC_\mu'(\tfrac{1}{z})
=-z^2G_{\tilde{k}}'(z)
=\int_{0}^{b}\frac{z^{2}}{(z-t)^{2}}k(t)\,dt
-\int_{-b}^{0}\frac{z^{2}}{(z-t)^{2}}k(t)\,dt 
\]
for $z$ in $\C^+$. Moreover, the assumptions on $k$ entail
the existence of the limits $k'(0+)$ and $k'(0-)$, since (with $b$
chosen as above)
\[
k'(x)=-\int_x^bk''(t)\,dt\longrightarrow -\int_0^bk''(t)\, dt 
\qquad\text{as $x\downarrow0$},
\]
and similarly  
\[
k'(x)=\int_{-b}^x k''(t)\,dt\longrightarrow \int_{-b}^0k''(t)\, dt 
\qquad\text{as $x\uparrow0$}.
\]
The same argument ensures the existence of the limits $k''(0+)$ and
$k''(0-)$. 
Hence, for $z=x+iy$ in $\C^+$, we can perform integration by parts
twice as follows:
\begin{align*}
\int_{0}^{b}\frac{z^{2}}{(z-t)^{2}}k(t)\,dt
&=z^{2}\left[\frac{k(t)}{z-t}\right]_{0}^{b}
-z^{2}\int_{0}^{b}\frac{k'(t)}{z-t}\,dt
\\[.2cm]
&=-zk(0+)-z^{2}\Big(
  \big[-\log(z-t)k'(t)\big]_{0}^{b}+\int_{0}^{b}\log(z-t)k''(t)\,dt\Big)
\\[.2cm]
&=-zk(0+)-z^2\log(z)k'(0+)-z^2\int_0^b\log(z-t)k''(t)\,dt,
\end{align*}
where $\log$ is the standard branch of the logarithm on
$\C\setminus\{iy\mid y\le0\}$. 
Here $-zk(0+)-z^2\log(z)k'(0+)\to0$, as $z\to0$, $z\in\C^+$.
For the last term note that $k''$ extends to a continuous
function on $[0,b]$, since the limit $k''(0+)$ exists in $\R$ as
mentioned above. Hence we can apply Lemma~\ref{lem1b}(iii) to establish that
\begin{equation*}
\begin{split}
\limsup_{z\to0 \atop z\in\C^+}\Big|\int_0^b\log(z-t)k''(t)\,dt\Big|
&=
\limsup_{z\to0 \atop z\in\C^+}
\Big|\tfrac{1}{2}\int_0^b\log((t-x)^2+y^2)k''(t)\,dt
+i\int_0^b\arg(x-t+iy)k''(t)\,dt\Big|
\\[.2cm]
&\leqq\Big|\tfrac{1}{2}\int_0^bk''(t)\log(t^2)\,dt\Big|+\|k''\|_\infty b\pi
\\[.2cm]
&\leqq\tfrac{1}{2}\|k''\|_{\infty}\int_0^b|\log(t)|\,dt+\|k''\|_\infty b\pi
<\infty,
\end{split}
\end{equation*}
so that $z^2\int_0^b\log(z-t)k''(t)\,dt\to0$, as $z\to0$,
$z\in\C^+$. We conclude that
$\int_{0}^{b}\frac{z^{2}}{(z-t)^{2}}k(t)\,dt\to0$ as $z\to0$,
$z\in\C^+$, and similar arguments show that
$\int_{-b}^{0}\frac{z^{2}}{(z-t)^{2}}k(t)\,dt\to0$
as $z\to0$, $z\in\C^+$. Thus we have established that
$z^2G_{\tilde{k}}(z)\to0$ as $z\to0$, $z\in\C^+$, and since the 
function $z\mapsto z^2G_{\tilde{k}}(z)$ is continuous on
$\C^+\cup(\R\setminus\{0\})$, this immediately implies that the same
convergence holds as $z\to0$, $z\in\C^+\cup(\R\setminus\{0\})$.
\end{proof}

The following lemma is a modification of Lemma~4.1 in \cite{HaTh2015}. For
completeness we include a full proof in the appendix.

\begin{lem}\label{approximation_lemma}
Let $k\colon\R\setminus\{0\}\to[0,\infty)$ be a
  function which is increasing on $(-\infty,0)$, decreasing on
  $(0,\infty)$ and such that
  $\frac{k(t)}{|t|}1_{\R\setminus\{0\}}(t)\,dt$ is a L\'evy
  measure. 
Then there exists a sequence $(k_n)$ of functions
  $k_n\colon\R\setminus\{0\}\to[0,\infty)$
  satisfying the following conditions for all $n$ in $\N$:

\begin{enumerate}[\rm(a)]

\item $k_n$ has bounded support.

\item $k_n\in C^\infty(\R\setminus\{0\})$, and $k_n$ and all its
  derivatives are bounded functions.

\item $k_n$ is increasing on $(-\infty,0)$ and decreasing on
  $(0,\infty)$.

\item
$\displaystyle{
\frac{|t|k_n(t)}{1+t^2}\,dt\overset{\rm w}{\longrightarrow}
\frac{|t|k(t)}{1+t^2}\,dt}$
\quad as $n\to\infty$.

\end{enumerate}
\end{lem}

With the preceding lemmas in place we are now ready to prove the
following characterization of the freely selfdecomposable
distributions on $\R$.
 
\begin{thm}\label{main1}
For a probability measure $\mu$ on $\R$ the following statements are
equivalent: 
\begin{enumerate}[{\rm (i)}]
\item\label{aa} $\mu\in L(\boxplus)$.
\item\label{bb} The free cumulant transform $\calC_{\mu}$ of $\mu$ extends to
  an analytic map $\calC_{\mu}: \C^{-}\to\C$, satisfying that 
$\im(\calC_{\mu}'(w))\leqq 0$ for any $w\in\C^{-}$.
\item\label{cc}
There exists $\xi$ in $\R$ and a measure $\rho$ on $\R$, satisfying that
$\int_{\R}\log(|x|+2)\,\rho(dx)<\infty$, such that
$\calC_\mu'$ can be extended to all of $\C^-$ via the formula:
\begin{align}\label{corform}
\calC_{\mu}'(w)=\xi + \int_{\R}\frac{x+w}{1-xw}\,\rho(dx), \qquad (w\in\C^{-}).
\end{align}
\end{enumerate}
If \eqref{aa}-\eqref{cc} are satisfied, then the pair $(\xi,\rho)$ in
(iii) is unique, and the free characteristic triplet for $\mu$ is
given by $(a,\eta, \frac{k(x)}{|x|} dx)$, where  
\begin{align}
a&= \frac{1}{2}\rho(\{0\}), \notag
\\
\eta&= \xi+\int_{\R}x\left(1_{[-1,1]}(x)-\frac{1-x^2}{(1+x^2)^2}
\right) \frac{k(x)}{|x|}\, dx, \notag
\\
 k(x)&= \begin{cases}
\int_{x}^{\infty}\frac{1+y^{2}}{y^{2}}\rho(dy), &\text{if $x>0$,}
\\[2mm]
\int_{-\infty}^{x}\frac{1+y^{2}}{y^{2}}\rho(dy), &\text{if $x<0$.}
\label{eq_free_char_triplet}
\end{cases} 
\end{align}
\end{thm}

\begin{rem} 
It is a bit unexpected that the condition \eqref{bb} implies in particular that $\mu \in I(\boxplus)$ and hence the condition in Theorem \ref{BV_char_of_FID}: $\im(\varphi_\mu(z)) \leq0$ for all $z$ in $\C^+$. We provide an interpretation of this implication in terms of free cumulants in Remark \ref{PD}.  
\end{rem}

\begin{proofof}{\bf Proof of Theorem~\ref{main1}.}
\eqref{aa} $\Rightarrow$ \eqref{bb}: Assume that $\mu\in L(\boxplus)$
  with free characteristic triplet $(\eta,a,\frac{k(x)}{|x|}dx)$,
  where $k$ is increasing on $(-\infty,0)$ and decreasing on
  $(0,\infty)$. Then note that (cf.\ \eqref{eqno1})
\begin{equation}
\im\big(\calC_{\mu}'(w)\big)=2a\, \im(w)
+\im\Big(\frac{d}{dw}\int_{\R}\Big(\frac{1}{1-tw}-1-tw1_{[-1,1]}(t)\big)
\frac{k(t)}{|t|}\,dt\Big), \qquad(w\in\C^-).
\label{pf_of_Main1_eq1}
\end{equation}
Since $a\,\im(w)\leqq0$ for any $w$ in $\C^-$, we may assume without loss
of generality that $a=0$. Furthermore, since the right hand side of
\eqref{pf_of_Main1_eq1} does not depend on $\eta$, it suffices to show
that there exists a real constant $\eta_0$, such that
$\im(\calC_{\mu_0}'(w))\leqq0$ for all $w$ in $\C^-$, where $\mu_0$ is the
measure with free characteristic triplet $(0,\eta_0,\frac{k(t)}{|t|}\,dt)$.

By Lemma~\ref{approximation_lemma} we can choose a sequence
$(k_n)_{n\in\N}$ of functions satisfying conditions (a)-(c) of that
lemma, and such that $\sigma_n(d t)\to\sigma(d t)$ weakly as
$n\to\infty$,
where $\sigma_n(dt)=\frac{|t|k_n(t)}{1+t^2}\,dt$ and
$\sigma(dt)=\frac{|t|k(t)}{1+t^2}\,dt$.
Then let $\mu_n$ and $\mu_0$ be the measures in $L(\boxplus)$
with free generating pairs (cf.\ \eqref{eqno1a} and \eqref{eqno1b})
$(0,\sigma_n)$ and $(0,\sigma)$, respectively. For any fixed $z$ in
$\C^+$ we then have that
\[
\varphi_{\mu_n}(z)=\int_{\R}\frac{1+tz}{z-t}\sigma_n(dt)
\xrightarrow[n\to\infty]{}
\int_{\R}\frac{1+tz}{z-t}\sigma(dt)=\varphi_{\mu_0}(z)
\]
and that
\[
\varphi_{\mu_n}'(z)=-\int_{\R}\frac{1+t^2}{(z-t)^2}\sigma_n(dt)
\xrightarrow[n\to\infty]{}
-\int_{\R}\frac{1+t^2}{(z-t)^2}\sigma(dt)=\varphi_{\mu_0}'(z),
\]
as the functions $t\mapsto\frac{1+tz}{z-t}$ and
$t\mapsto\frac{1+t^2}{(z-t)^2}$ are both continuous and bounded on $\R$. This
further implies that
\[
\calC_{\mu_n}'(\tfrac{1}{z})=\varphi_{\mu_n}(z)
-z\varphi_{\mu_n}'(z)\xrightarrow[n\to\infty]{}
{\cal C}_{\mu_0}'(\tfrac{1}{z})
\]
for any $z$ in $\C^+$. By Lemma~\ref{lem3} 
we have that
$\im(\calC_{\mu_n}'(w))\leqq0$ for any $w$ in $\C^-$ and $n$ in $\N$,
and hence also $\im(\calC_{\mu_0}'(w))\leqq0$ for any $w$ in
$\C^-$. Since $\mu_0$ has free characteristic triplet
$(0,\eta_0,\frac{k(t)}{|t|})$ for some real constant $\eta_0$, we have
established the necessary condition described above.

\eqref{bb} $\Rightarrow$ \eqref{cc}: Assume that \eqref{bb} is satisfied.
Then the function $z\mapsto -\calC_{\mu}'\left(\frac{1}{z}\right)$ is
analytic from $\C^{+}$ into $\C^{+}\cup\R$, and hence by
Nevanlinna-Pick representation (see e.g.\ \cite[Formula (3.3)]{Akh1965}) 
there exist $c$ in $[0,\infty)$, $\xi$ in $\R$ and a finite measure
$\rho$ on $\R$ such that  
\begin{align*}
-\calC_{\mu}'\Big(\frac{1}{z}\Big) = c z-\xi +\int_{\R}
\frac{1+xz}{x-z}\,\rho(dx), \qquad(z\in\C^+).
\end{align*}
Then
\begin{align*}
\calC_{\mu}'(w)=-\frac{c}{w}+\xi+\int_{\R}\frac{x+w}{1-xw}\,\rho(dx), 
\qquad(w\in\C^{-}),
\end{align*}
and it remains to establish that $c=0$ and that
$\int_{\R}\log(|x|+2)\,\rho(dx)<\infty$. For $y$ in $(0,\infty)$ we
note that
\begin{align*}
\calC_{\mu}(-iy)
=\calC_{\mu}(-i)-i\int_{1}^{y}\calC_{\mu}'(-it)\,dt,
\end{align*}
so that
\begin{align*}
\re\left(\calC_{\mu}(-iy)\right)&=\re\left(\calC_{\mu}(-i)\right)+\im
\left(\int_{1}^{y}\calC_{\mu}'(-it)\,dt\right)
\\
&=\re\left( \calC_{\mu}(-i)\right)-\int_{1}^{y}
\left(\frac{c}{t}+\int_{\R}\frac{t(1+x^{2})}{1+t^{2}x^{2}}
\,\rho(dx)\right)\,dt
\\    
&=\re\left(\calC_{\mu}(-i)\right)
+c\log\big(\tfrac{1}{y}\big)
+\int_{\R}\Big(\int_{y}^{1}\frac{t(1+x^{2})}{1+t^{2}x^{2}}\,dt\,\Big)\rho(dx)
\\
&=\re\left( \calC_{\mu}(-i)\right)+c\log\big(\tfrac{1}{y}\big)
+\frac{1}{2}\int_{\R}
\frac{1+x^{2}}{x^{2}}\log\Big(\frac{1+x^{2}}{1+x^{2}y^{2}}\Big)\,\rho(dx). 
\end{align*}
Since $\varphi_{\mu}(iv)=o(v)$ as $v\to\infty$ (see
Bercovici-Voiculescu \cite[Proposition~5.6]{BeVo1993}), 
it follows that $\lim_{y\downarrow 0} \calC_{\mu}(-iy)=
-\lim_{y\downarrow 0}iy\varphi(iy^{-1})=0$.
On the other hand the monotone convergence theorem yields that
\begin{align*}
\lim_{y\downarrow 0}\int_{\R}
\frac{1+x^{2}}{x^{2}}\log\Big(\frac{1+x^{2}}{1+x^{2}y^{2}}\Big)\,\rho(dx) 
=\int_{\R}\frac{1+x^{2}}{x^{2}}\log(1+x^{2})\,\rho(dx)\in [0,\infty].
\end{align*}
We thus conclude that
\begin{align*}
0=\re \left(\calC_{\mu}(-i)\right)+c\cdot\infty+\int_{\R}
\frac{1+x^{2}}{x^{2}}\log(1+x^{2})\,\rho(dx).
\end{align*}
As a result, we obtain that $c=0$ and
$\int_{\R}\log(|x|+2)\,\rho(dx)<\infty$, as desired.

\eqref{cc} $\Rightarrow$ \eqref{aa}: Assume that \eqref{cc} holds, and
note initially that this implies that $\calC_\mu$ also
extends to an analytic function on all of $\C^-$.
Next let $2a=\rho(\{0\})$ and let $z$ be a fixed point
in $\C^-$. Then we denote by $[-i,z]$ the straight line from $-i$ to
$z$, and by $\int_{-i}^z{\cal C}_\mu'(\omega)\,d\omega$ the path
integral of ${\cal C}_\mu'$ along $[-i,z]$. Since $z,-i\in\C^-$, it is
standard to check that 
\[
\sup\big\{\big|\tfrac{x+\omega}{1-x\omega}\big|\bigm| x\in\R, \
\omega\in[-i,z]\big\}<\infty, 
\]
and hence we may apply Fubini's Theorem in the following calculation:
\begin{align*}
\calC_{\mu}(z)&=\calC_{\mu}(-i)+\int_{-i}^{z}\calC_{\mu}'(\omega)\,d\omega&\\
		      &=\calC_{\mu}(-i)+\int_{-i}^{z}\Big(\xi
                        +\int_{\R} 
                      \frac{x+\omega}{1-x\omega}\,\rho(dx) \Big)\,d\omega&\\
		      &=\calC_{\mu}(-i)
		      +\int_{-i}^{z}\Big(\xi +2a\omega 
		      + \int_{\R\backslash \{0\}}
                      \frac{x+\omega}{1-x\omega}\,\rho(dx) \Big)\,d\omega&\\  
		      &=\calC_{\mu}(-i)
		      +\xi(z+i)+a(z^{2}-i^{2})
		      +\int_{\R\backslash \{0\}}\Big(
		      \int_{-i}^{z}\frac{x+\omega}{1-x\omega}\,d\omega\Big)\,
                      \rho(dx)&\\ 
		      &=\calC_{\mu}(-i)+i\xi +a
		      +\xi z +az^{2}+\int_{\R\backslash \{0\}}
                      \Big[\Big(-\log(1-x\omega)-\frac{x\omega}{1+x^{2}}\Big)
                      \Big]_{\omega=-i}^{\omega=z}\frac{1+x^{2}}{x^{2}}\,\rho(dx),&    
\end{align*}
where $\log$ is the standard branch of the logarithm on
$\C\setminus(-\infty,0]$.  
By second order Taylor expansion, it follows that $\log(1-\omega
x)=-\omega x-\frac{1}{2}\omega^2x^2+o(x^2)$, and therefore
\begin{equation}
\log(1-\omega x)+\frac{x\omega}{1+x^2}
=-\frac{\omega x^3}{1+x^2}-\tfrac{1}{2}\omega^2x^2+o(x^2),
\quad\text{as $x\to0$,} 
\label{pf_of_Main1_eq2}
\end{equation}
for each fixed $\omega$ in $\C^-$.
This implies that
\[
\int_{[-1,1]\setminus\{0\}}
\Big|\log(1+xi)-\frac{xi}{1+x^{2}}\Big|\frac{1+x^2}{x^2}\,\rho(d x)
<\infty,
\]
and since also
\begin{align*}
\int_{\R\setminus[-1,1]}
\Big|\log(1+xi)-\frac{xi}{1+x^{2}}\Big|
\frac{1+x^2}{x^2}\,\rho(d x)
&\leqq\int_{\R\setminus[-1,1]}2|\log(1+xi)|\,\rho(d x)
+\int_{\R\setminus[-1,1]}\frac{1}{|x|}\,\rho(dx)
\\
&\leqq\int_{\R\setminus[-1,1]}2\big(\log(1+|x|)+\pi\big)\,\rho(d x)
+\rho(\R\setminus[-1,1])<\infty,
\end{align*} 
by the assumptions on $\rho$, it follows that the integral
$\int_{\R\setminus\{0\}}
(\log(1+xi)-\frac{xi}{1+x^{2}})\frac{1+x^2}{x^2}\,\rho(d x)$
is a well-defined complex number. We thus conclude that
\begin{equation}
{\cal C}_\mu(z)=A+\xi z+a z^2
+\int_{\R\setminus\{0\}}\Big(-\log(1-xz)-\frac{xz}{1+x^{2}}\Big)
\frac{1+x^2}{x^2}\,\rho(d x),
\label{pf_of_Main1_eq3}
\end{equation}
where 
$A={\cal C}_\mu(-i)+i\xi+a
+\int_{\R\setminus\{0\}}(\log(1+xi)-\frac{xi}{1+x^{2}})
\frac{1+x^2}{x^2}\,\rho(d x)$. 

Now we put
\begin{align*}
k(x):=
\begin{cases}
\int_{x}^{\infty}\frac{1+y^{2}}{y^{2}}\rho(dy), &\text{if $x>0$}\\
\int_{-\infty}^{x}\frac{1+y^{2}}{y^{2}}\rho(dy), &\text{if $x<0$}.
\end{cases}
\end{align*}
Then for any continuity points $r,s$ of $\rho$, such that $0<r<s$, we find by
application of Lemma~\ref{lem2} that
\begin{align}
\int_r^s&\Big(-\log(1-xz)-\frac{xz}{1+x^{2}}\Big)\frac{1+x^{2}}{x^{2}}
\,\rho(dx)\notag
\\ 
&=\Big[\Big(\log(1-xz)+\frac{xz}{1+x^{2}}\Big)k(x)\Big]_{r}^{s}
+\int_r^s\Big(\frac{z}{1-xz}-z\frac{d}{dx}\Big(\frac{x}{1+x^{2}}
\Big)\Big)k(x)\,dx \label{pf_of_Main1_eq4}
\\
&=\Big[\Big(\log(1-xz)+\frac{xz}{1+x^{2}}\Big)k(x)\Big]_{r}^{s}
+\int_r^s\Big(\frac{xz}{1-xz}-z\frac{x-x^3}{(1+x^2)^2}\Big)
\frac{k(x)}{x}\,dx.\notag
\end{align}
Here Lemma~\ref{lem4}\eqref{lem4-2} entails that
\begin{equation}
\Big[\Big(\log(1-xz)+\frac{xz}{1+x^{2}}\Big)k(x)\Big]_{r}^{s}
\longrightarrow0,\qquad\text{as $r\downarrow0$ and $s\uparrow\infty$.}
\label{pf_of_Main1_eq5}
\end{equation}
Note further that
\begin{equation}
\int_r^s\Big(\frac{xz}{1-xz}-z\frac{x-x^3}{(1+x^2)^2}\Big)
\frac{k(x)}{x}\,dx
=\int_{r}^s\Big(\frac{1}{1-xz}-1-xz1_{[-1,1]}(x)\Big)\frac{k(x)}{x}\,dx
+\int_r^sg_z(x)\frac{k(x)}{x}\,dx,
\label{pf_of_Main1_eq6}
\end{equation}
where
\[
g_z(x)=\frac{zx}{1-xz}-\frac{1}{1-xz}+1+xz1_{[-1,1]}(x)
-z\frac{x-x^3}{(1+x^2)^2}=zx\Big(1_{[-1,1]}(x)-\frac{1-x^2}{(1+x^2)^2}\Big).
\]
Since $g_z(x)\to0$ as $|x|\to\infty$ and $g_z(x)=o(x^2)$ as $x\to0$, and
since $\frac{k(x)}{x}\,dx$ is a L\'evy measure (cf.\
Lemma~\ref{lem4}), it follows that
$\int_0^{\infty}|g_z(x)|\frac{k(x)}{x}\,dx<\infty$.

Considering now sequences $(r_n)$ and $(s_n)$ of continuity points for
$\rho$, such that $r_n\to0$ and $s_n\to\infty$ as $n\to\infty$, it
follows by combining \eqref{pf_of_Main1_eq4}-\eqref{pf_of_Main1_eq6}
that
\begin{align*}
\int_0^\infty\Big(-&\log(1-xz)-\frac{xz}{1+x^{2}}\Big)\frac{1+x^{2}}{x^{2}}
\,\rho(dx)\\
&=\lim_{n\to\infty}\int_{r_n}^{s_n}
\Big(-\log(1-xz)-\frac{xz}{1+x^{2}}\Big)\frac{1+x^{2}}{x^{2}}
\,\rho(dx)
\\
&=z\int_0^{\infty}x\Big(1_{[-1,1]}(x)-\frac{1-x^2}{(1+x^2)^2}\Big)
\frac{k(x)}{x}\,dx
+\int_{0}^\infty\Big(\frac{1}{1-xz}-1-xz1_{[-1,1]}(x)\Big)\frac{k(x)}{x}\,dx.
\end{align*}
By similar arguments, it follows that
\begin{align*}
\int_{-\infty}^0\Big(-&\log(1-xz)-\frac{xz}{1+x^{2}}\Big)\frac{1+x^{2}}{x^{2}}
\,\rho(dx)\\
&=z\int_{-\infty}^0x\Big(1_{[-1,1]}(x)-\frac{1-x^2}{(1+x^2)^2}\Big)
\frac{k(x)}{|x|}\,dx
+\int_{-\infty}^0\Big(\frac{1}{1-xz}-1-xz1_{[-1,1]}(x)\Big)\frac{k(x)}{|x|}\,dx,
\end{align*}
and combining these two formulas with \eqref{pf_of_Main1_eq3}, we obtain
the expression:
\begin{equation}
{\cal C}_\mu(z)=A+\eta z+az^2
+\int_{\R}\Big(\frac{1}{1-xz}-1-xz1_{[-1,1]}(x)\Big)\frac{k(x)}{|x|}\,dx,
\end{equation}
where 
$\eta=\xi+\int_{\R}x(1_{[-1,1]}(x)-\frac{1-x^2}{(1+x^2)^2})\frac{k(x)}{|x|}\,dx$.
Finally, let $\mu'$ be the
measure in $L(\boxplus)$ with free characteristic triplet
$(a,\eta,\frac{k(x)}{|x|}\,dx)$. Then by two applications of
\cite[Proposition~5.6]{BeVo1993} we find that
$
0=\lim_{y\uparrow0}{\cal C}_\mu(iy)
=A+\lim_{y\uparrow0}{\cal C}_{\mu'}(iy)=A.
$
Thus $\mu=\mu'\in L(\boxplus)$, and this completes the proof.
\end{proofof}

Before stating the following corollary to Theorem~\ref{main1} we recall
that for a compactly supported probability measure $\mu$ on $\R$ the
$R$-transform $\calR_\mu$ can be extended analytically to an
open neighborhood of 0. Thus $\calC_\mu(z)=z\calR_\mu(z)$ admits a power series
expansion: 
\begin{equation}
\calC_\mu(z)=\sum_{n=1}^{\infty}\kappa_{n}(\mu)z^n
\label{eq_power_series_exp}
\end{equation}
in a ball around 0, and the coefficients $\{\kappa_n(\mu)\mid n\ge1\}$
are the free cumulants of $\mu$ (see e.g.\ \cite{BG2006}).
For a general measure $\mu$ on $\R$
with moments of all orders the free cumulants are defined from the
moments via M\"obius inversion (see \cite{NiSpBook}) and
\eqref{eq_power_series_exp} only holds as an asymptotic expansion (see
\cite{BG2006}). Recall that a sequence $\{a_n\}_{n=1}^\infty$ of real numbers is said to be \emph{conditionally positive definite} if the $N\times N$ matrix $\{a_{i+j}\}_{i,j=1}^{N}$ is positive definite for any $N \geq1$ (see \cite{NiSpBook}).  

\begin{cor}\label{sd_vs_cumulants}
Let $\mu$ be a Borel probability measure on $\R$ with moments of all
orders, and let $\{\kappa_{n}(\mu)\}_{n=1}^{\infty}$ be the free cumulant
sequence of $\mu$. Then the following statements hold:
 
\begin{enumerate}[{\rm (i)}] 

\item \label{BB}
	If $\mu$ is FSD then $\{n\kappa_{n}(\mu)\}_{n=1}^{\infty}$ is
        conditionally positive definite.  

\item \label{aaa}
	Suppose further that $\mu$ has compact support. Then $\mu$
        is FSD if and only if $\{n\kappa_{n}(\mu)\}_{n=1}^{\infty}$ 
        is conditionally positive definite. 

\end{enumerate}
\end{cor}

\begin{rem}\label{rem:exp_growth}
A Borel probability measure $\mu$ on $\R$ with finite moments of all
orders is compactly supported if and only if the sequence
$\{\kappa_n(\mu)\}_{n=1}^\infty$ does not grow faster than
exponentially; i.e., there exists $c>0$ such that $|\kappa_n(\mu)|
\leq c^n$ for all $n\geq1$. This is also equivalent to the property
that the L\'evy measure of $\mu$ is compactly supported. See
\cite[Lemma 13.13, Proposition 13.15]{NiSpBook}.     
\end{rem}

\begin{rem}\label{PD}
Suppose that $\mu$ is compactly supported. It is well known that $\mu$ is in $I(\boxplus)$ if and only if $\{\kappa_n(\mu)\}_{n=1}^\infty$ is conditionally positive definite (see e.g.\ \cite[Theorem 13.16]{NiSpBook}). Our result then shows the implication 
\[
\text{$\{n\kappa_{n}(\mu)\}_{n=1}^{\infty}$ is conditionally positive definite} \quad \Longrightarrow \quad \text{$\{\kappa_{n}(\mu)\}_{n=1}^{\infty}$ is conditionally positive definite}. 
\]
This implication can be proved more directly from the following two facts: The sequence $\{\frac{1}{n}\}_{n=1}^\infty$ is conditionally positive definite since 
$\frac{1}{n}$ is the $(n-1)$-th moment of the uniform distribution on $(0,1)$; the product of two conditionally positive definite sequences is again conditionally positive definite. 

\end{rem}

\begin{proofof}{\bf Proof of Corollary~\ref{sd_vs_cumulants}.}
\eqref{BB} \ From Theorem 1.3 in \cite{BG2006}, the asymptotic expansion of
the free cumulant transform exists up to any order, and then according
to Lemma~A.1 in \cite{BG2006}, the equation 
\begin{align}\label{AS}
\calC_{\mu}'(z) = \sum_{n=0}^{\infty}(n+1)\kappa_{n+1}(\mu)z^{n}.
\end{align}
holds in the sense of
an asymptotic expansion. Applying the second part of \cite[Theorem
3.2.1]{Akh1965} to the function $-\calC_\mu'(1/z)+\kappa_1(\mu)$ (which maps $\C^+$ into $\C^+\cup\R$ by Theorem \ref{main1}) then
implies that the sequence $\{(n+2)\kappa_{n+2}(\mu)\}_{n=0}^{\infty}$
is a moment sequence of a finite measure.

\eqref{aaa} \ The sufficiency is already proved in \eqref{BB}, so it suffices to show the necessity. This proof is similar to the discussion in \cite[Chapter
13]{NiSpBook}.  Suppose that  $\{n\kappa_{n}(\mu)\}_{n=1}^{\infty}$ is
conditionally positive definite. Since $\{n\kappa_n(\mu)\}_{n\geq1}$ does not grow
faster than exponentially, Proposition 13.14 in \cite{NiSpBook}
yields the existence of a finite measure $\widetilde{\rho}$ on $\R$
with compact support such that  
\begin{align*} 
(n+2)\kappa_{n+2}(\mu) = \int_\R x^{n}\,\widetilde{\rho}(dx),\qquad (n\geq 0).
\end{align*}
Therefore, for all $z$ with sufficiently small absolute value, we have
that (cf.\ \eqref{eq_power_series_exp})
\begin{align*}
\calC_{\mu}'(z) 
&= \kappa_{1}(\mu) +  \sum_{n=1}^{\infty}(n+1)\kappa_{n+1}(\mu)z^{n} 
= \kappa_{1}(\mu) +\sum_{n=1}^{\infty} \left(\int_{\R}
  x^{n-1}\widetilde{\rho}(dx)\right) z^{n}&\\ 
&= \kappa_{1}(\mu) +\sum_{n=1}^{\infty} \left(\int_\R
  (x^{n+1}+x^{n-1})\frac{\widetilde{\rho}(dx)}{1+x^{2}}\right)
z^{n}.&\\ 
&= \kappa_{1}(\mu)  - \int_\R x \,\rho(d x)+
\sum_{n=0}^{\infty}\int_\R x^{n+1} z^n \,\rho(d x) +
\sum_{n=0}^{\infty} \int_\R x^{n} z^{n+1} \,\rho(d x)  & \\ 
&=  \kappa_{1}(\mu) - \int_\R x \,\rho(d x) +\int_\R \frac{x+z}{1-xz}
\rho(dx), 
\end{align*}
where we put $\rho(dx)=\frac{\widetilde{\rho}(dx)}{1+x^{2}}$. From the
resulting expression of this calculation it follows that $\calC_\mu'$
extends to an analytic function on $\C^-$, and hence a correctly
chosen anti-derivative of $\calC_\mu'$ is an analytic extension of
$\calC_\mu$ to all $\C^-$. Since $\calC_\mu'$ has the form \eqref{corform} it follows from Theorem
\ref{main1} that $\mu$ is freely selfdecomposable. 
\end{proofof}

\section{Free selfdecomposability of the normal 
distribution}\label{sec:Proof_of_Main2}

In this section we prove that the classical normal (or Gaussian)
distributions belong to the class $L(\boxplus)$ of freely
selfdecomposable probability distributions, and more generally that the
Askey-Wimp-Kerov distributions $\mu_c$ belong to $L(\boxplus)$ for all
$c$ in $[-1,0]$. Apart from Theorem~\ref{main1} the proof is based on
results from Belinschi et al.\ \cite{BBLS2011} and the following Theorem due to Kerov (see
\cite[Theorem~8.2.5]{Ke1998}).

\begin{thm}\label{Kerovs_Thm}
For any $c$ in $(-1,\infty)$ there exists a probability measure
$\tau_c$ on $\R$, such that the following relation holds between the
Cauchy-Stieltjes transforms:
\[
-\frac{d}{d z}\log G_{\mu_c}(z)=G_{\tau_c}(z), \qquad(z\in\C^+).
\]
\end{thm}

As a final preparation we introduce the class $\calUI$ consisting of
those Borel probability measures on $\R$ for which
there exists a simply connected domain $\Omega$ in $\C$, such that
$\Omega\supset\C^{+}$ and such that the reciprocal Cauchy-Stieltjes
transform $F_{\mu}$ can be extended to an analytic bijection $F_{\mu}\colon
\Omega \to \C^{+}$. If $\mu$ is in $\calUI$, then it is FID, as was
proved in \cite{ArHa2013}. For distributions in $\calUI$
Theorem~\ref{main1} then yields the following characterization of free
selfdecomposability:

\begin{lem}\label{lem5}
Let $\mu$ be a measure in $\calUI$ with domain $\Omega$ as described above. Then the following statements are
equivalent: 
\begin{enumerate}[{\rm (i)}]
\item
$\mu\in L(\boxplus)$.
\item
$\im\big(\omega -\frac{F_{\mu}(\omega)}{F_{\mu}'(\omega)}\big)\leqq 0$ 
\ for all $\omega$ in $\Omega$.
\end{enumerate}
\end{lem}

\begin{proof}
By the definition of the free cumulant transform (see
\eqref{def_Cmu_eq}) and analytic continuation we have that
$\calC_{\mu}(z) = z F_{\mu}^{-1}\left(\frac{1}{z}\right)-1$ for all
$z$ in $\C^-$. Setting $\omega=F_{\mu}^{-1}(\frac{1}{z})\in\Omega$, we
then get that
\begin{align*}
\calC_{\mu}'(z) =
F_{\mu}^{-1}\Big(\frac{1}{z}\Big)-\frac{1}{z}(F_{\mu}^{-1})'
\Big(\frac{1}{z}\Big) 
=\omega -\frac{F_{\mu}(\omega)}{F_{\mu}'(\omega)}.
\end{align*}
Since $F_\mu\colon\Omega\to\C^+$ is a bijection, condition (ii) in the
lemma is thus equivalent to the condition that $\im(\calC_\mu'(z))\leqq0$
for all $z$ in $\C^-$. According to Theorem~\ref{main1} the latter
condition is, in turn, equivalent to (i) in the lemma. 
\end{proof}

\begin{thm}\label{main2}
For any $c$ in $[-1,0]$ the Askey-Wimp-Kerov distribution $\mu_c$ is
freely selfdecomposable. 
\end{thm}

\begin{proof} When $c=-1$, $\mu_c$ is a Dirac measure and the theorem is
  trivial. So let $c$ be a fixed number in $(-1,0]$. According to the
  proof of Theorem~3.1 in \cite{BBLS2011} we have that
  $\mu_c\in\mathcal{UI}$, so the 
reciprocal Cauchy transform $F_{\mu_c}$ extends to an analytic bijection
$F_{\mu_c}\colon\Omega\to\C^+$ defined on some region $\Omega$
containing $\C^+$ (and depending on $c$).
According to Lemma~\ref{lem5} we then have to establish that
\[
\im\Big(z-\frac{F_{\mu_c}(z)}{F_{\mu_c}'(z)}\Big)\leqq0
\]
for any $z$ in $\Omega$.
We consider first $z$ in $\C^+$ and observe that
\[
\frac{F_{\mu_c}'(z)}{F_{\mu_c}(z)}= G_{\mu_c}(z)\frac{d}{dz}\Big(\frac{1}{G_{\mu_c}(z)}\Big)
=-\frac{G_{\mu_c}'(z)}{G_{\mu_c}(z)}=-\frac{d}{dz}\log G_{\mu_c}(z), 
\]
where $\log$ is the standard branch of the logarithm on
$\C\setminus\{iy\mid y\geqq0\}$.
Thus, according to Kerov's Theorem (Theorem~\ref{Kerovs_Thm}), there
exists a probability measure $\tau_c$ on $\R$, such that
\[
\frac{F_{\mu_c}'(z)}{F_{\mu_c}(z)}=G_{\tau_c}(z), 
\quad\text{or equivalently}\quad
\frac{F_{\mu_c}(z)}{F_{\mu_c}'(z)}=F_{\tau_c}(z),
\qquad(z\in\C^+).
\]
This implies in particular that
\[
\im\Big(z-\frac{F_{\mu_c}(z)}{F_{\mu_c}'(z)}\Big)
=\im\big(z-F_{\tau_c}(z)\big)<0, \qquad(z\in\C^+),
\]
where the inequality follows from Corollary~5.3 in \cite{BeVo1993}.

Next consider $\omega$ in $\Omega\setminus\C^+$. Then according to 
formula (3.5) in \cite{BBLS2011} we have that
\[
\frac{F_{\mu_c}'(\omega)}{F_{\mu_c}(\omega)}
=\omega-F_{\mu_c}(\omega)-\frac{c}{F_{\mu_c}(\omega)},
\]
and therefore
\[
\im\Big(\frac{F_{\mu_c}'(\omega)}{F_{\mu_c}(\omega)}\Big)
=\im(\omega)-\im\big(F_{\mu_c}(\omega)\big)-c\, \im\big(1/F_{\mu_c}(\omega))\leqq0,
\]
since $\im(\omega)\leqq0$ and $-c\geqq0$, and since
$F_{\mu_c}(\omega)\in\C^+$, so that $1/F_{\mu_c}(\omega)\in\C^-$.
It follows that
\[
\im\Big(\omega-\frac{F_{\mu_c}(\omega)}{F_{\mu_c}'(\omega)}\Big)
=\im(\omega)-\im\Big(\frac{1}{F_{\mu_c}'(\omega)/F_{\mu_c}(\omega)}\Big)
\leqq0,
\]
and this completes the proof.
\end{proof}

\begin{cor}
For any $\xi$ in $\R$ and $\sigma$ in $(0,\infty)$, the normal
distribution $N(\xi,\sigma^2)$ is freely selfdecomposable.
\end{cor}

\begin{proof}
If $\xi=0$ and $\sigma^2=1$, this corresponds to the case $c=0$ in
Theorem~\ref{main2}. The general case subsequently follows from the
fact that $L(\boxplus)$ is closed under scalings and translations.
\end{proof}

\begin{rem}
Let $\xi$ and $\sigma$ be a real and a positive number, respectively.
For any $t$ in $(0,\infty)$ the probability measure
$N(\xi,\sigma^2)^{\boxplus t}$ may be defined as the law at time $t$
of a free L\'evy process $(X_t)$ such that $X_1$ has law
$N(\xi,\sigma^2)$. In particular the free L\'evy measure
for $N(\xi,\sigma^2)^{\boxplus t}$ is $t\nu$, with $\nu$ the free
L\'evy measure of $N(\xi,\sigma^2)$, and hence
$N(\xi,\sigma^2)^{\boxplus t}$ is FSD as well. In particular this
implies that $N(\xi,\sigma^2)^{\boxplus t}$ is unimodal (cf.\
\cite{HaTh2015}).
\end{rem}

\appendix

\section{Proofs of various technical lemmas}

\begin{proofof}{\bf Proof of Lemma~\ref{lem1b}.}

(i) \ We initially put
$
\|f\|_{\infty}:=\sup_{x\in[a,b]}|f(x)|<\infty.
$
For $u$ in $[a,b]$ and $v$ in $(0,1)$ we then have that
\begin{align*}
\Big|\int_{a}^{b}f(x)\log((x-u)^2+v^2)\,dx
&-\int_{a}^{b}f(x)\log((x-u)^2)\,dx\Big|
\\[.2cm]
&=\Big|\int_{a-u}^0f(x+u)\log\big(\tfrac{x^2+v^2}{x^2}\big)\,dx
+\int_0^{b-u}f(x+u)\log\big(\tfrac{x^2+v^2}{x^2}\big)\,dx\Big|
\\[.2cm]
&\leqq\int_{a-b}^0\|f\|_{\infty}\log\big(\tfrac{x^2+v^2}{x^2}\big)\,dx
+\int_0^{b-a}\|f\|_{\infty}\log\big(\tfrac{x^2+v^2}{x^2}\big)\,dx
\\[.2cm]
&=2\|f\|_{\infty}\int_0^{b-a}\log\big(\tfrac{x^2+v^2}{x^2}\big)\,dx. 
\end{align*}
The resulting expression does not depend on $u$ and tends to 0 as $v\downarrow0$ by dominated convergence with the dominating function $\log(\tfrac{x^2+1}{x^2})$ for all $v\in(0,1)$. 


(ii) \ Recall that $\arg$ denotes the standard continuous argument
function on $\C\setminus\{iy\mid y\leqq0\}$, and therefore
\[
\int_{a}^{b}f(x)\arg((u-x)+iv)\,dx
=\int_{a}^{u}f(x)\arctan\big(\tfrac{v}{u-x}\big)\,dx
+\int_{u}^{b}f(x)\big(\pi-\arctan\big(\tfrac{v}{x-u}\big)\big)\,dx
\]
for any $u$ in $[a,b]$ and $v$ in $[0,\infty)$. Note further that
\[
\Big|\int_{a}^{u}f(x)\arctan\big(\tfrac{v}{u-x}\big)\,dx\Big|
\leqq\int_0^{u-a}\|f\|_{\infty}\big|\arctan\big(\tfrac{v}{y}\big)\big|\,dy
\leqq
\|f\|_{\infty}\int_0^{b-a}\big|\arctan\big(\tfrac{v}{y}\big)\big|\,dy,
\]
and similarly
\begin{align*}
\Big|\int_{u}^{b}f(x)\big(\pi
-\arctan\big(\tfrac{v}{x-u}\big)\big)\,dx
-\pi\big(F(b)-F(u)\big)\Big|
&\leqq\int_0^{b-u}\|f\|_{\infty}\big|\arctan\big(\tfrac{v}{y}\big)\big|\,dy
\\[.2cm]
&\leqq \|f\|_{\infty}\int_0^{b-a}\big|\arctan\big(\tfrac{v}{y}\big)\big|\,dy.
\end{align*}
Since $\int_0^{b-a}|\arctan\big(\tfrac{v}{y}\big)|\,dy$
does not depend on $u$ and converges to 0, as $v\downarrow0$, by
dominated convergence, the estimates above verify (ii).

(iii) \ If $0\notin[a,b]$, then
$\sup_{x\in[a,b]}|\log((u-x)^2+v^2)|<\infty$ for $|u+iv|$ small enough,
and hence the assertion follows by dominated convergence. We may thus
assume in the following that $0\in[a,b]$, and we establish only that
\[
\int_0^bf(x)\log((u-x)^2+v^2)\,dx\longrightarrow
\int_0^bf(x)\log(x^2)\,dx \qquad\text{as $u+iv\to0$ from $\C^+$,}
\]
since it follows symmetrically that 
$\int_a^0f(x)\log((u-x)^2+v^2)\,dx\to\int_a^0f(x)\log((u-x)^2+v^2)\,dx$
as $u+iv\to0$ from $\C^+$.
For $u$ in $(0,\infty)$ we note first that
\begin{equation}
\begin{split}
\Big|\int_0^bf(x)\log((x-u)^2&+v^2)\,dx-\int_0^bf(x)\log(x^2)\,dx\Big|
\\[.2cm]
&=\Big|\int_{-u}^{b-u}f(x+u)\log(x^2+v^2)\,dx-\int_0^bf(x)\log(x^2)\,dx\Big|
\\[.2cm]
&\leqq\Big|\int_{-u}^0f(x+u)\log(x^2+v^2)\,dx\Big|
+\Big|\int_0^{b-u}\big(f(x+u)-f(x)\big)\log(x^2+v^2)\,dx\Big|
\\
&\phantom{\leqq\Big|\int_{-u}^0}
+\Big|\int_0^{b-u}f(x)\big(\log(x^2+v^2)-\log(x^2)\big)\,dx\Big| 
+\Big|\int_{b-u}^bf(x)\log(x^2)\,dx\Big|.
\label{lem1b_eq0}
\end{split}
\end{equation}
Assuming henceforth that $u^2+v^2\leqq\tfrac{1}{2}\wedge b$, we have here that
\begin{equation}
\Big|\int_{-u}^0f(x+u)\log(x^2+v^2)\,dx\Big|
\leqq\|f\|_\infty\int_{-u}^0-\log(x^2+v^2)\,dx
\leqq-\|f\|_\infty\int_{-u}^0\log(x^2)\,dx\longrightarrow0,
\label{lem1b_eq1}
\end{equation}
as $u\downarrow0$, by dominated convergence. Similarly we find that
\begin{equation}
\Big|\int_{b-u}^bf(x)\log(x^2)\,dx\Big|\leqq\|f\|_\infty\int_{b-u}^b|\log(x^2)|\,dx
\longrightarrow0, \qquad\text{as $u\downarrow0$.}
\label{lem1b_eq2}
\end{equation}
We note further that
\begin{equation}
\begin{split}
\Big|\int_0^{b-u}&\big(f(x+u)-f(x)\big)\log(x^2+v^2)\,dx\Big|
\leqq\Big(\sup_{x\in[0,b-u]}\big|f(x+u)-f(x)\big|\Big)
\int_0^{b}|\log(x^2+v^2)|\,dx
\\[.2cm]
&\leqq\Big(\sup_{x\in[0,b-u]}\big|f(x+u)-f(x)\big|\Big)
\Big(\int_0^{b\wedge\frac{1}{2}}-\log(x^2)\,dx
+\int_{b\wedge\frac{1}{2}}^b|\log(b^2\wedge\tfrac{1}{4})|\vee
|\log(b^2+\tfrac{1}{2})|\,dx\Big)
\\[.2cm]
&\longrightarrow 0, \qquad\text{as $u\downarrow0$},
\label{lem1b_eq3}
\end{split}
\end{equation}
since the supremum goes to 0 as $u\downarrow0$, by uniform continuity
of $f$, and since both integrals in the resulting expression are finite.
Note finally that
\begin{equation}
\Big|\int_0^{b-u}f(x)\big(\log(x^2+v^2)-\log(x^2)\big)\,dx\Big| 
\leqq\|f\|_\infty\int_0^b\log\big(\tfrac{x^2+v^2}{x^2}\big)\,dx
\longrightarrow 0, \qquad\text{as $v\downarrow0$},
\label{lem1b_eq4}
\end{equation}
since the integral in the resulting expression goes to 0 as
$v\downarrow0$ as seen in the proof of (i). Combining
\eqref{lem1b_eq0}-\eqref{lem1b_eq4}, it follows that  
$\int_0^bf(x)\log((u-x)^2+v^2)\,dx\to\int_0^bf(x)\log(x^2)\,dx$
as $u+iv\to0$ from $(0,\infty)+i(0,\infty)$. A similar argumentation
establishes the same convergence when $u+iv\to0$ from
$(-\infty,0)+i(0,\infty)$. 
\end{proofof}

\begin{proofof}{\bf Proof of Lemma~\ref{lem4}.}
\eqref{lem4-1} \ We must show that  $\int_{\R}(1\wedge x^2)\frac{k(x)}{|x|}\,dx<\infty$.
We note first that
\begin{align*} 
\int_{-1}^1x^{2}\frac{k(x)}{|x|}dx =\int_0^1 x k(x)\,dx
+\int_{-1}^0 |x|k(x)\,dx,
\end{align*}
where, by Tonelli's Theorem,
\begin{align*}
\int_0^1 x k(x)\,dx
=\int_{0}^{1}\Big(\int_{x}^{\infty}x\frac{1+y^{2}}{y^{2}}\rho(dy)\Big)\,dx
=\int_{0}^{\infty}\frac{1+y^{2}}{y^{2}}\Big(\int_{0}^{y \wedge 1}x\,dx\Big)\,\rho(dy)
=\int_{0}^{\infty} \frac{(y \wedge 1)^{2}}{2y^{2}}(1+y^{2})\rho(dy)<\infty,
\end{align*}
since $\rho$ is a finite measure, and the function $y\mapsto\frac{(y
  \wedge 1)^{2}}{2y^{2}}(1+y^{2})$ is bounded on $(0,\infty)$. 
In the same manner, $\int_{-1}^0|x|k(x)\,dx<\infty$.
Note next that
\begin{align*}
\int_{\R\setminus[-1,1]}\frac{k(x)}{|x|}dx=\int_1^{\infty}\frac{k(x)}{x}dx 
+\int_{-\infty}^{-1}\frac{k(x)}{|x|}dx,
\end{align*}
where
\begin{align*}
\int_{1}^{\infty}\frac{k(x)}{x}dx 
=\int_{1}^{\infty}\frac{1}{x}\Big(\int_{x}^{\infty}
\frac{1+y^{2}}{y^{2}}\,\rho(dy)\Big)\,dx
=\int_{1}^{\infty}\frac{1+y^{2}}{y^{2}}
\Big(\int_{1}^{y}\frac{1}{x}\,dx\Big)\,\rho(dy)
=\int_{1}^{\infty} \log(y)\frac{1+y^{2}}{y^{2}}\,\rho(dy)
<\infty,
\end{align*}
since the function $y\mapsto\frac{1+y^{2}}{y^{2}}$ is bounded on
$[1,\infty)$, and $\int_{1}^{\infty}\log(y)\,\rho(dy)<\infty$.
In a similar way it follows that
$\int_{-\infty}^{-1}\frac{k(x)}{|x|}dx<\infty$, and this completes the proof
of \eqref{lem4-1}.

(ii) \ 
For any $\varepsilon$ in $(0,\infty)$, there exists $\delta$ in
$(0,1)$, such that 
$\int_{(0,\delta]}(1+y^{2})\rho(dy)\leqq2\rho((0,\delta])\leqq\varepsilon$.
Since $\rho$ is finite we have that
$\int_{\delta}^{\infty}y^{-2}(1+y^2)\,\rho(dy)<\infty$, and we can thus
choose $\gamma$ in $(0,\infty)$, such that 
$\gamma^2\int_{\delta}^{\infty}y^{-2}(1+y^2)\,\rho(dy)\leqq\epsilon$.
Now for any $x$ in $(0,\delta\wedge\gamma)$ we find that
\begin{align*}
x^{2}k(x)	
=x^{2}\int_{x}^{\delta}\frac{1+y^{2}}{y^{2}}\rho(dy)
+x^{2}\int_{\delta}^{\infty}\frac{1+y^{2}}{y^{2}}\rho(dy)
\leqq \int_{x}^{\delta}y^{2}\frac{1+y^{2}}{y^{2}}\rho(dy)
+\gamma^2\int_{\delta}^{\infty}\frac{1+y^{2}}{y^{2}}\rho(dy)
\leqq2\epsilon, 
\end{align*}
and this shows that $x^2k(x)\to0$ as $x\downarrow0$. 
In a similar way, it follows that $x^2k(x)\to0$ as $x\uparrow0$, and
this completes the proof of (ii).

(iii) \ For $x$ in $[1,\infty)$ we note first that
\begin{align*}
|\log(x)k(x)| = \int_{x}^{\infty} \log(x)\frac{1+y^{2}}{y^{2}}\,\rho(dy)
\leqq\int_{x}^{\infty} \log(y)\frac{1+y^{2}}{y^{2}}\,\rho(dy)
\leqq\int_{x}^{\infty}2\log(y)\,\rho(dy)\longrightarrow 0,
\quad\text{as $x\to\infty$},
\end{align*}
since $\int_1^{\infty}\log(y)\,\rho(dy)<\infty$.
Similarly it follows that $\log(|x|)k(x)\to0$ as $x\to-\infty$. 

(iv) \ Recall that here $\log$ denotes the standard branch of the
logarithm on $\C\setminus(-\infty,0]$, and let $\arg$ denote the
corresponding argument function. For $z=u+iv$ in $\C^-$ we then have that
\begin{align*}
\Big|\log(1-zx)+\frac{xz}{1+x^2}\Big|k(x)
&=\Big|\tfrac{1}{2}\log((1-ux)^2+v^2x^2)+i\arg((1-ux)+ivx)
+\frac{xz}{1+x^2}\Big|k(x)
\\
&\leqq\tfrac{1}{2}\Big(\log(x^2)+\log((x^{-1}-u)^2+v^2)\Big)k(x)
+\Big(\pi+\frac{|xz|}{1+x^2}\Big)k(x),
\end{align*}
where the resulting expression tends to 0 as $|x|\to\infty$ by (iii).

By second order Taylor expansion we note next that
$\log(1-zx)=-zx-\frac{1}{2}z^2x^2+o(x^2)$, and therefore
\begin{equation*}
\log(1-zx)+\frac{xz}{1+x^2}
=-\frac{zx^3}{1+x^2}-\tfrac{1}{2}z^2x^2+o(x^2), \quad\text{as $x\to0$.}
\end{equation*}
Consequently,
\begin{equation*}
\Big(\log(1-zx)+\frac{xz}{1+x^2}\Big)k(x)
=\Big(-\frac{zx}{1+x^2}-\tfrac{1}{2}z^2+o(1)\Big)x^2k(x)
\longrightarrow0, \quad\text{as $x\to0$,}
\end{equation*}
by (ii). This completes the proof of (iv) and hence the proof of the
lemma.
\end{proofof}

\begin{proofof}{\bf Proof of Lemma~\ref{lem1}.} 
We consider initially the case $m=1$ and arbitrary $a',b'$ such that
$a<a'<b'<b$. It suffices then to show that $G_f$ can be extended to a
continuous function on $\C^+\cup(a',b')$.
For any $z$ in $\C^+$ we have that
\begin{align*}
G_{f}(z)=\int_{a}^{a'} \frac{f(x)}{z-x}\,dx
+\int_{a'}^{b'} \frac{f(x)}{z-x}\,dx+\int_{b'}^{b} \frac{f(x)}{z-x}\,dx
=:G_{1}(z)+G_{2}(z)+G_{3}(z).
\end{align*}
It is clear that $G_{1}$ and $G_{3}$ can be extended to analytic functions on 
$\C^{+}\cup (a',b')\cup\C^-$, and it remains then 
to prove that $G_{2}$ can be extended
to a continuous function on $\C^{+}\cup (a',b')$.
In the following we denote by $\log$ the standard continuous
branch of the logarithm on $\C\setminus\{iy\mid y\leqq0\}$.
Using integration by parts, we then obtain for $z=u+iv$ in
$\C^+$ that
\begin{align}
G_{2}(z)
=-f(b')\log(u+iv-b')+f(a')\log(u+iv-a')+\int_{a'}^{b'}f'(x)\log(u+iv-x)\,dx.
\label{lem1_eq1}
\end{align}
Here the first and second terms
$-f(b')\log(u+iv-b')+f(a')\log(u+iv-a')$ are analytic  
on $\C^{+}\cup(a',b')$
with respect to $z=u+iv$. Regarding the integral in \eqref{lem1_eq1}
an application of Lemma~\ref{lem1b}(i)-(ii) yields that
\begin{align*}
\int_{a'}^{b'}f'(x)\log(u+iv-x)\,dx
&=\tfrac{1}{2}\int_{a'}^{b'}f'(x)\log((x-u)^{2}+v^{2})\,dx
+i\int_{a'}^{b'}f'(x)\arg (u+iv-x)\,dx\\
&\longrightarrow \int_{a'}^{b'}f'(x)\log(|x-u|)\,dx+i\pi(f(b')-f(u))
\end{align*}
as $v\downarrow0$, uniformly w.r.t.\ $u\in(a',b')$.
From this it follows readily that $G_{2}$ can be extended to a
continuous function on $\C^{+}\cup(a',b')$, where
\begin{align*}
G_{2}(u) &= -f(b')\log(u-b')+f(a')
\log(u-a')+\int_{a'}^{b'}f'(x)\log(|x-u|)\,dx+i\pi(f(b')-f(u))
\\[.2cm]
&=-f(b')\log(b'-u)+f(a')\log(u-a')+\int_{a'}^{b'}f'(x)\log(|x-u|)\,dx
-i\pi f(u)
\end{align*}
for $u$ in $(a',b')$.

Suppose next that $m\geqq 2$, and that $f\in C^{m}((a,b))$. With $a',b'$
and $G_2$ as above, it suffices to show that the derivatives
$G_2',G_2'',\ldots,G_2^{(m-1)}$ can be extended to continuous functions
on $\C^+\cup(a',b')$. 
For any $n$ in $\{1,\ldots,m-1\}$ it follows by
induction and integration by parts that
\begin{align*}
G_2^{(n)}(z)
=\sum_{k=0}^{n-1}(n-1-k)!(-1)^{n-k}\left[
\frac{f^{(k)}(x)}{(z-x)^{n-k}}\right]_{x=a'}^{x=b'} 
+\int_{a'}^{b'}\frac{f^{(n)}(x)}{z-x}\,dx.
\end{align*}
From this expression and the preceding part of the proof, it follows
readily that $G_2',\ldots,G_2^{(m-1)}$ can be extended to continuous
functions on $\C^+\cup(a',b')$, as desired.   
\end{proofof}

\begin{proofof}[Proof of Lemma~\ref{approximation_lemma}.]
For each $n$ in $\N$ we introduce first the function
$k_n^0\colon(0,\infty)\to[0,\infty)$ defined by
\[
k_n^0(t)=
\begin{cases}
k(\tfrac{1}{n}), &\text{if $t\in(0,\frac{1}{n})$}\\
k(t), &\text{if $t\in[\frac{1}{n},n]$}\\
0, &\text{if $t\in(n,\infty)$},
\end{cases}
\]
and we note that $k_n^0\leqq k_{n+1}^0$ for all $n$.
Next we choose a non-negative function $\varphi$ from $C^{\infty}_c(\R)$, such that 
$\textrm{supp}(\varphi)\subseteq[-1,0]$, and
$\int_{-1}^0\varphi(t)\,dt=1$. We then define the function
$R_n\colon(0,\infty)\to[0,\infty)$ as the convolution
\begin{equation}
R_n(t)=n\int_{-1/n}^0k_n^0(t-s)\varphi(ns)\,ds
=\int_0^1k_n^0(t+\tfrac{u}{n})\varphi(-u)\,du, \qquad(t\in(0,\infty)).
\label{eq8}
\end{equation}
Since $k_n^0$ is a bounded, decreasing function, it follows
immediately from \eqref{eq8} that so is $R_n$. 
Moreover, $\textrm{supp}(R_n)\subseteq(0,n]$ by the
definition of $k_n^0$. Note also that
\[
R_n(t)=n\int_{0}^n\varphi(n(t-s))k_n^0(s)\,ds, \qquad(t\in(0,\infty)).
\]
Since $k_n^0$ as well as the derivatives of $\varphi$ are bounded
functions, it follows then by differentiation under the integral sign that
$R_n$ is a $C^{\infty}$-function on $(0,\infty)$ with \emph{bounded}
derivatives given by 
\[
R_n^{(p)}(t)=n^{p+1}\int_{0}^n\varphi^{(p)}(n(t-s))k_n^0(s)\,ds,
\qquad(p\in\N, \ t\in(0,\infty)).
\]
By dominated convergence it follows further for any $p$ in $\N_0$ that
\[
\lim_{t\downarrow0}R_n^{(p)}(t)=n^{p+1}\int_{0}^n\varphi^{(p)}(-ns)k_n^0(s)\,ds\in\R.
\]
For any $t$ in $(0,\infty)$ and $n$ in $\N$ note next that
\[
R_n(t)\leqq\int_0^1k_{n+1}^0(t+\tfrac{u}{n})\varphi(-u)\,du
\leqq\int_0^1k_{n+1}^0(t+\tfrac{u}{n+1})\varphi(-u)\,du
=R_{n+1}(t).
\]
Moreover, the monotonicity assumptions imply that $k$ is continuous at
almost all $t$ in $(0,\infty)$ (with respect to Lebesgue measure). For
such a $t$ we further consider $n$ so
large that $t+\frac{u}{n}\in[\frac{1}{n},n]$ for all $u$ in
$[0,1]$. For such $n$ it follows then that
\[
R_n(t)=\int_0^1k(t+\tfrac{u}{n})\varphi(-u)\,du
\xrightarrow[n\to\infty]{}
\int_0^1k(t)\varphi(-u)\,du=k(t)
\]
by monotone convergence. We conclude that $R_n(t)\uparrow k(t)$ as
$n\to\infty$ for almost all $t$ in $(0,\infty)$. 

Applying the considerations above to the function
$\kappa\colon(0,\infty)\to[0,\infty)$ given by $\kappa(t)=k(-t)$, it
follows that we can construct a sequence 
$(L_n)_{n\in\N}$ of non-negative functions defined on $(-\infty,0)$
and with the following properties:

\begin{itemize}

\item For all $n$ in $\N$ the function $L_n$ has bounded support.

\item For all $n$ in $\N$ we have that $L_n\in C^{\infty}((-\infty,0))$,
  and $L_n^{(p)}$ is bounded for all $p$ in $\N_0$.

\item For all $n$ in $\N$ the function $L_n$ is increasing on
  $(-\infty,0)$.

\item $L_n(t)\uparrow k(t)$ as $n\to\infty$ for almost all $t$ in
  $(-\infty,0)$ (with respect to Lebesgue measure).

\end{itemize}

We are now ready to define $k_n\colon\R\setminus\{0\}\to[0,\infty)$ by
\[
k_n(t)=
\begin{cases}
R_n(t), &\text{if $t>0$,}\\
L_n(t), &\text{if $t<0$.}
\end{cases}
\]
It is then apparent from the argumentation above that $k_n$ satisfies
the conditions (a)-(c) in the lemma, and it remains to
show that $\frac{|t|k_n(t)}{1+t^2}\,dt\to \frac{|t|k(t)}{1+t^2}\,dt$ weakly as $n\to\infty$. But for any
bounded continuous function $g\colon\R\to\R$ we find that
\begin{equation*}
\begin{split}
\int_{\R}g(t)\frac{|t|k_n(t)}{1+t^2}\,dt
&=
\int_{-\infty}^0g(t)\frac{|t|L_n(t)}{1+t^2}\,dt
+\int_0^{\infty}g(t)\frac{tR_n(t)}{1+t^2}\,dt
=
\int_{-\infty}^0g(t)\frac{|t|L_n(t)}{1+t^2}\,dt
+\int_0^{\infty}g(t)\frac{tR_n(t)}{1+t^2}\,dt
\\[.2cm]
&\xrightarrow[n\to\infty]{}
\int_{-\infty}^0g(t)\frac{|t|k(t)}{1+t^2}\,dt
+\int_0^{\infty}g(t)\frac{tk(t)}{1+t^2}\,dt
=
\int_{\R}g(t)\frac{|t|k(t)}{1+t^2}\,dt,
\end{split}
\end{equation*}
where, when letting $n\to\infty$, we used dominated convergence on
each of the integrals; note in particular that
$\frac{|t|L_n(t)}{1+t^2}$ and $\frac{tR_n(t)}{1+t^2}$ are dominated
almost everywhere by $\frac{|t|k(t)}{1+t^2}$ on the relevant
intervals, and here $\int_{\R}\frac{|t|k(t)}{1+t^2}\,dt<\infty$, since
$\frac{k(t)}{|t|}\,dt$ is a L\'evy measure. This completes the proof.
\end{proofof}

\subsection*{Acknowledgments}

T.\ Hasebe was supported by JSPS Grant-in-Aid for Young Scientists (B)
15K17549. N.\ Sakuma was supported by JSPS Grant-in-Aid for Scientific
Research (C) 15K04923. Part of this work was done while the second
author was visiting The University of Aarhus.  He sincerely appreciates the hospitality of The University of Aarhus. 
The authors are grateful to Octavio Arizmendi for a useful suggestion.


\begin{thebibliography}{9999}

\bibitem{Akh1965}
Akhiezer, Naum~I. {\it The Classical Moment Problem}. Oliver \& Boyd Ltd
(1965).

\bibitem{An2003}
Anshelevich, Michael. 
Free martingale polynomials. 
{\it J. Funct. Anal}. {\bf 201} (2003), no. 1, 228--261.

\bibitem{AsWi1984} Askey, Richard\ and Wimp, Jet. Associated Laguerre and
  Hermite polynomials. {\it Proc.\ Royal Soc.~Edinburgh} {\bf 96A}
  (1984), 15--37.

\bibitem{ArHa2013}
Arizmendi, Octavio and Hasebe, Takahiro. 
On a class of explicit Cauchy-Stieltjes transforms related to monotone
stable and free Poisson laws.  
{\it Bernoulli} {\bf 19} (2013), no. 5B, 2750--2767.

\bibitem{BNTh2002}
Barndorff-Nielsen, Ole E. and Thorbj{\o}rnsen, Steen.
Self-decomposability and L{\'e}vy processes in free probability. 
{\it Bernoulli} {\bf 8} (2002), no. 3, 323--366. 

\bibitem{BNTh2004}
Barndorff-Nielsen, Ole E. and Thorbj{\o}rnsen, Steen.
A connection between free and classical infinite divisibility.
{\it Infin. Dimens. Anal. Quantum Probab. Relat. Top.} {\bf 7} (2004), no. 4, 573--590. 


\bibitem{BBLS2011}
Belinschi, Serban T., Bo{\.z}ejko, Marek, Lehner, Franz and Speicher, Roland.
The normal distribution is $\boxplus$-infinitely divisible.
{\it Adv. Math.} {\bf 226} (2011), no. 4, 3677--3698. 

\bibitem{BG2006}
Benaych-Georges, Florent. 
Taylor expansions of R-transforms: application to supports and moments. 
{\it Indiana Univ. Math. J.} {\bf 55} (2006), no. 2, 465--481.

\bibitem{BeVo1993}
Bercovici, Hari and Voiculescu, Dan.
Free convolution of measures with unbounded support. 
{\it Indiana Univ. Math. J.} {\bf 42} (1993), no. 3, 733--773.

\bibitem{BePa1999}
Bercovici, Hari and Pata, Vittorino.
Stable laws and domains of attraction in free probability theory.
With an appendix by Philippe Biane. 
{\it Ann. of Math. (2)} {\bf 149} (1999), no. 3, 1023--1060. 

\bibitem{BB2006}
Bo{\.z}ejko, Marek and Bryc, W{\l}odzimierz. 
On a class of free L{\'e}vy laws related to a regression problem. 
{\it J. Funct. Anal.} {\bf 236} (2006), no. 1, 59--77.

\bibitem{CG2008} 
Chistyakov, Gennadii P. and Goetze, Friedrich. 
Limit theorems in free probability theory. I.
{\it Ann. Probab.} {\bf 36}, no. 1 (2008), 54--90.

\bibitem{Fe} Feller, William. {\it An introduction to probabililty
    theory and its applications} II, second edition. Wiley Series in Probability and
  Mathematical Statistics, John Wiley \& Sons, New York (1971).

\bibitem{GK68} Gnedenko, Boris Vladimirovich and Kolmogorov, Andrei
  Nikolaevich.  
{\it Limit Distributions for Sums of Independent Random Variables}. 
Addison-Wesley Publishing Company, Inc., Reading, Mass.-London-Don
Mills., Ont. (1968).


\bibitem{HT7} 
Haagerup, Uffe and Thorbj{\o}rnsen, Steen. 
On the free gamma distributions. 
{\it Indiana Univ. Math.\ J.} {\bf 63}
  (2014), no.~4, 1159--1194.


\bibitem{HaSa2016}
Hasebe, Takahiro and Sakuma, Noriyoshi.
Unimodality for free L{\'e}vy processes. 
{\it  Ann. Inst. Henri Poincar{\'e} Probab. Stat.} {\bf53} (2017), no. 2, 916--936.

\bibitem{HaTh2015}
Hasebe, Takahiro and Thorbj{\o}rnsen, Steen.
Unimodality of the freely selfdecomposable probability laws. 
 {\it J. Theoret. Probab.} {\bf 29} (2016), no. 3, 922--940. 

\bibitem{Ke1998} Kerov, Sergei. 
Interlacing measures. In: Kirillovs
  Seminar on Representation Theory, Amer.\ Math.\ Soc.\ Transl.\ vol.\
  181, Amer.\ Math.\ Soc.\ (1998), pp.~35-83.
  
\bibitem{NiSpBook}
Nica, Alexandru and Speicher, Roland.
{\it Lectures on the combinatorics of free probability}. 
London Mathematical Society Lecture Note Series, {\bf 335}. 
Cambridge University Press, Cambridge, 2006. xvi+417
 
\bibitem{PAS}
P{\'e}rez-Abreu, V\'ictor and Sakuma, Noriyoshi.
Free generalized gamma convolutions. 
{\it Electron. Commun. Probab.} {\bf 13} (2008), 526--539.

\bibitem{SaYo2001}
Saitoh, Naoko and Yoshida, Hiroaki.
The infinite divisibility and orthogonal polynomials with a constant recursion formula in free probability theory.
{\it Probab. Math. Statist}. {\bf 21} (2001), no. 1, Acta Univ. Wratislav. No. 2298, 159--170.

\bibitem{Sato99}
Sato, Ken-iti.
{\it L{\'e}vy processes and infinitely divisible distributions.}
Translated from the 1990 Japanese original. Revised by the author. 
Cambridge Studies in Advanced Mathematics, {\bf 68}. Cambridge University Press, Cambridge, 1999. xii+486 pp. 

\bibitem{VDN}
Voiculescu, Dan, Dykema, Ken and Nica, Alexandru.
{\it Free random variables.} 
A noncommutative probability approach to free products with applications to random matrices, operator algebras and harmonic analysis on free groups. CRM Monograph Series, 1. American Mathematical Society, Providence, RI, 1992. vi+70 pp.
\end{thebibliography}
\end{document}